\documentclass{article}

 \PassOptionsToPackage{numbers, compress}{natbib}
\usepackage[final]{nips_2016}

\usepackage[utf8]{inputenc} 
\usepackage[T1]{fontenc}    
\usepackage[colorlinks,citecolor=blue,urlcolor=blue,linkcolor=blue,linktocpage=true]{hyperref}       
\usepackage{url}            
\usepackage{booktabs}       
\usepackage{amsfonts}       
\usepackage{nicefrac}       
\usepackage{microtype}      

\usepackage{amsmath,amssymb,bbm,color}
\usepackage{cases}
\usepackage{graphicx}
\usepackage{caption}
\usepackage{subcaption}
\usepackage{tabularx}
\usepackage{microtype}
\usepackage{enumitem}
\usepackage{url}
\usepackage[usenames,dvipsnames]{xcolor}
\usepackage[amsmath,amsthm,thmmarks]{ntheorem}
\usepackage{algpseudocode,algorithm}

\usepackage{cleveref}
\crefformat{equation}{(#2#1#3)}
\crefrangeformat{equation}{(#3#1#4) to~(#5#2#6)}
\crefname{equation}{}{}
\Crefname{equation}{}{}

\newtheorem{theorem}{Theorem}
\newtheorem{lemma}[theorem]{Lemma}

\newtheorem{definition}{Definition}
\crefname{definition}{\textbf{definition}}{definitions}
\Crefname{definition}{Definition}{Definitions}
\crefname{assumption}{\textbf{assumption}}{assumptions}
\Crefname{assumption}{Assumption}{Assumptions}

\let\hat\widehat
\let\tilde\widetilde

\newcommand{\dmm}{d_{\textsc{MM}}}
\newcommand{\dm}{d_{\textsc{M}}}

\newcommand{\estdens}{\hat{p}_h}
\newcommand{\conc}{\mathcal{C}}
\newcommand{\level}{\lambda}
\newcommand{\treef}{T_f}
\newcommand{\treeg}{T_g}
\newcommand{\treep}{T_p}
\newcommand{\esttree}{T_{\hat{p}_h}}
\newcommand{\treepstar}{T_{\densstar}}
\newcommand{\treeq}{T_q}
\newcommand{\colltf}{\{T_f\}}
\newcommand{\colltg}{\{T_g\}}
\newcommand{\cluster}{C}
\newcommand{\treeph}{T_{p_{h}}}

\newcommand{\ellinf}{\ell_\infty}
\newcommand{\life}{\mathrm{life}}

\newcommand{\densstar}{p_0}
\newcommand{\mPstar}{\mathbb{P}_0}

\renewenvironment{proof}{{\bf Proof.}}{$\Box$}

\title{Statistical Inference for Cluster Trees}

\author{
Jisu Kim\\
Department of Statistics\\
Carnegie Mellon University\\
Pittsburgh, USA \\
\texttt{jisuk1@andrew.cmu.edu} \\
\And
Yen-Chi Chen\\
Department of Statistics\\
University of Washington\\
Seattle, USA\\
\texttt{yenchic@uw.edu} \\
\And
Sivaraman Balakrishnan\\
Department of Statistics\\
Carnegie Mellon University\\
Pittsburgh, USA \\
\texttt{siva@stat.cmu.edu} \\
\And
Alessandro Rinaldo    \\
Department of Statistics\\
Carnegie Mellon University\\
Pittsburgh, USA \\
\texttt{arinaldo@stat.cmu.edu} \\
\And
Larry Wasserman\\
Department of Statistics\\
Carnegie Mellon University\\
Pittsburgh, USA \\
\texttt{larry@stat.cmu.edu} 
}

\begin{document}

\maketitle

\begin{abstract}
	A cluster tree provides a highly-interpretable
	summary of a density function by representing the hierarchy of its high-density clusters.
	It is estimated using the empirical tree, which is the cluster tree constructed
	from
	a density estimator.
	This paper 
	addresses the basic question of quantifying our 
	uncertainty by assessing
	the statistical significance of topological features
	of an empirical cluster tree. 
	We first study a variety of metrics that can be used
	to compare different trees, analyze their properties and 
	assess their suitability for inference.
	We then propose methods to construct 
	and summarize confidence
	sets for the unknown true cluster tree.
	We introduce a partial ordering on cluster
	trees which we use to prune some of
	the statistically insignificant features of the
	empirical tree, yielding interpretable and
	parsimonious cluster trees. 
	Finally, we illustrate the proposed
	methods on a variety of synthetic examples
	and furthermore demonstrate their utility in the analysis
	of a Graft-versus-Host Disease (GvHD) data set. 
\end{abstract}

\section{Introduction}

Clustering is a central problem in the analysis and exploration of data.
It is a broad topic, with several existing distinct 
formulations, objectives, and methods. Despite the extensive literature
on the topic, a common aspect of the clustering methodologies that has hindered its widespread
scientific adoption
is the dearth of methods for statistical inference
in the context of clustering. 
Methods for inference broadly allow us
to quantify our uncertainty, to discern ``true'' clusters from finite-sample
artifacts, as well as to rigorously test hypotheses related to the estimated
cluster structure. 

In this paper, we study statistical inference for the
\emph{cluster tree} of an unknown density. 
We assume that 
we observe an i.i.d. sample $\{X_1, \ldots, X_n\}$ 
from a distribution $\mPstar$
with unknown density $\densstar$. 
Here, $X_i\in{\cal X}\subset\mathbb{R}^d$.
The connected components
$\conc(\level)$, of the upper level set $\{x: \densstar(x) \geq \level\}$,
are called \emph{high-density clusters}. The set of high-density clusters
forms a nested hierarchy which is referred to as the \emph{cluster tree}\footnote{It is also 
	referred to as the density tree or the level-set tree.} of $\densstar$, which we denote
as $\treepstar$.

Methods for density clustering
fall broadly in the space of hierarchical clustering algorithms, and inherit
several of their advantages: they allow for extremely general cluster
shapes and sizes, and in general do not require the pre-specification
of the number of clusters. Furthermore, unlike flat clustering methods, hierarchical
methods are able to provide a multi-resolution summary of the underlying density.
The cluster tree, irrespective of the dimensionality of the input
random variable, is displayed as a two-dimensional object and this makes it
an ideal tool to visualize data.
In the context of statistical 
inference, density clustering has another important
advantage over other clustering methods: the object of inference, the cluster
tree of the unknown density $\densstar$, is clearly specified.

In practice, the cluster tree is estimated from a finite sample, $\{X_1,\ldots,X_n\} \sim \densstar$.
In a scientific application, we are often most interested in reliably distinguishing topological features genuinely present in the cluster tree of the unknown $\densstar$, from topological features that arise due to random fluctuations in the finite sample $\{X_1,\ldots,X_n\}$.
In this paper, we focus our inference on the 
cluster tree of the kernel density estimator, $\esttree$, where 
$\estdens$ is the 
kernel density estimator,
\begin{align}
\label{eq:kde}
\estdens(x) = \frac{1}{nh^d}\sum_{i=1}^n K\left(\frac{\|x-X_i\|}{h}\right),
\end{align}
where $K$ is a kernel and $h$ is an appropriately chosen bandwidth
\footnote{We address computing the tree $\esttree$, 
	and the choice of bandwidth in more detail 
	in what follows.}.


To develop methods for statistical inference on cluster trees, we construct a confidence set for $\treepstar$, i.e. a collection of trees that will include $\treepstar$ with some (pre-specified) probability. A confidence set can be converted to a hypothesis test, and a confidence set shows both statistical and scientific significances while a hypothesis test can only show statistical significances \cite[p.155]{Wasserman2010}.

To construct and understand the 
confidence set, we need to solve a 
few technical and conceptual 
issues. The first issue is that we need a 
\emph{metric} on 
trees, in order to quantify the collection of trees that
are in some sense ``close enough'' to $\esttree$ to be
statistically indistinguishable from it.
We use the bootstrap to construct tight data-driven confidence sets. 
However, only some metrics are sufficiently ``regular'' 
to be amenable to bootstrap inference, which guides our choice of a suitable metric 
on trees.

On the basis of a finite 
sample, the true density is indistinguishable from a density with additional
infinitesimal perturbations. This leads to the second
technical issue which is that our confidence set invariably contains infinitely
complex trees. Inspired by the idea of one-sided inference \cite{donoho1988one},
we propose a partial ordering on the set of all density trees
to define simple trees.
To find simple representative trees in the confidence set, we prune the empirical cluster tree
by removing statistically insignificant features. These pruned trees are valid with statistical
guarantees that are simpler than the empirical cluster tree in the proposed partial ordering.

{\bf Our contributions: } 
We begin by 
considering a variety of metrics on trees, 
studying their properties and discussing their suitability for
inference. 
We then propose a method of constructing confidence sets and for visualizing trees in this set.
This distinguishes aspects of the estimated 
tree correspond to real features (those present in the cluster tree $\treepstar$) 
from noise features. 
Finally, 
we apply our methods to several simulations, 
and a Graft-versus-Host Disease (GvHD) data set to demonstrate the usefulness of our techniques and the role of statistical inference in clustering problems.

{\bf Related work: }  
There is a vast literature on density trees (see for instance the
book by \citet{klemela2009smoothing}), and we focus our 
review on works most closely aligned with our paper. 
The formal definition of the cluster tree, and notions of consistency
in estimation of the cluster tree date back to the work of \citet{Hartigan81}.
Hartigan studied the efficacy of single-linkage in estimating the cluster tree
and showed that single-linkage is inconsistent when the input
dimension $d > 1$. Several fixes to single-linkage have since been proposed (see for instance 
\cite{stuetzle2010generalized}). The paper of \citet{chaudhuri2010rates} provided
the first rigorous
minimax analysis of the density clustering and 
provided a computationally tractable, consistent estimator 
of the cluster tree. The papers \cite{
	balakrishnan2012,chaudhuri14pruning,eldridge2015beyond,kpotufe2011pruning} 
propose various modifications and analyses of estimators for the cluster tree.
While the question of estimation has been extensively addressed, to our knowledge
our paper is the first concerning inference for the cluster tree. 

There is a literature on inference for phylogenetic trees (see
the papers \cite{felsenstein1985confidence,efron96}), but the object of inference
and the hypothesized generative models are typically quite different. 
Finally, in our paper, we also consider various metrics on trees. There are several
recent works, in the computational topology literature, that have considered 
different metrics on trees. The most relevant to our own work, are the papers
\cite{bauer2015strong,morozov2013interleaving} that propose the functional 
distortion metric and the interleaving distance on trees. These metrics, however,
are NP-hard to compute in general. In Section~\ref{sec:metric}, we 
consider a variety of computationally tractable metrics and assess their
suitability for inference.

\begin{figure}
	\center
	\includegraphics[height=1 in]{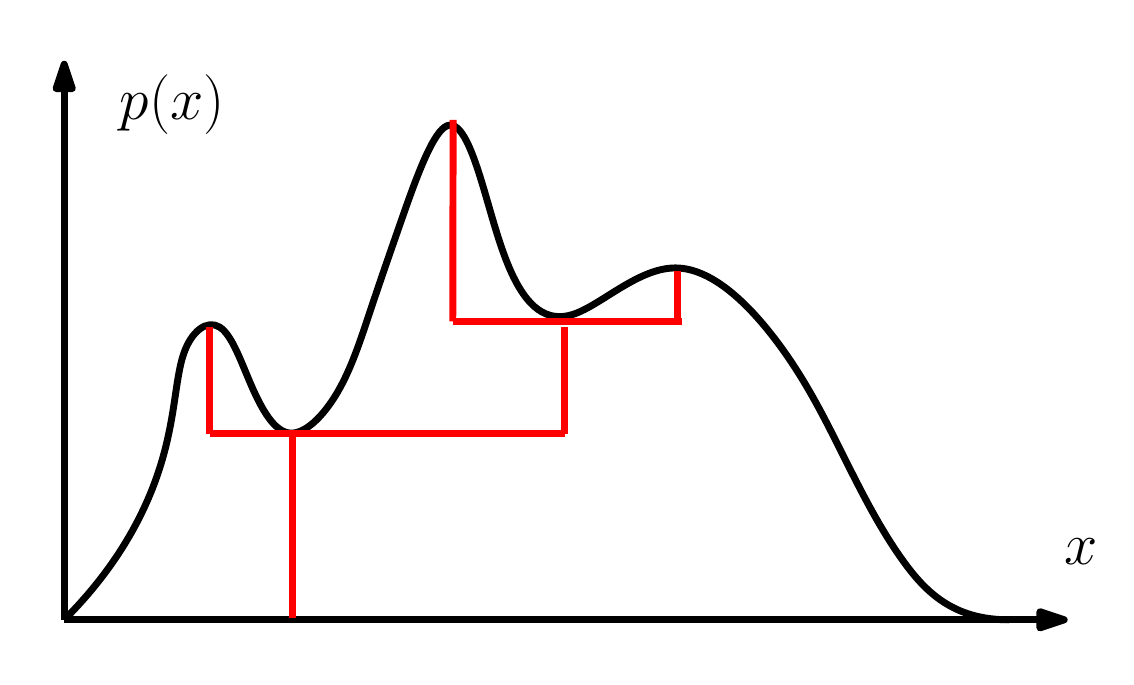}
	\includegraphics[height=1 in]{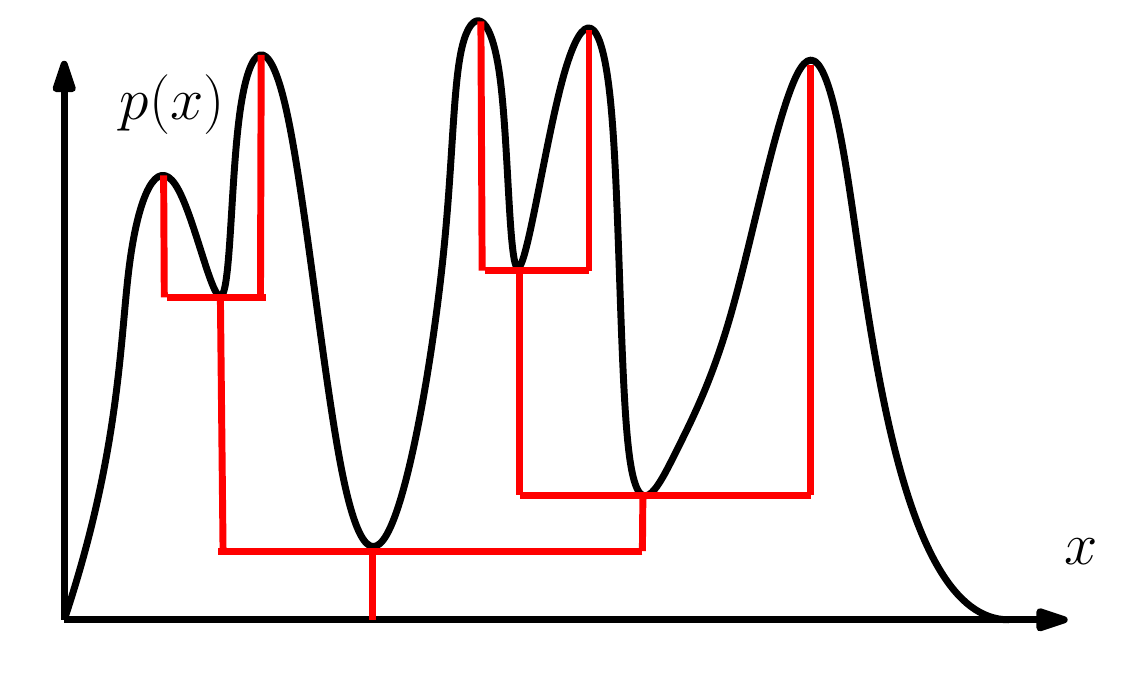}
	\caption{Examples of density trees. Black curves are the original density functions and the red trees are the
		associated density trees.}
	\label{fig:ex}
\end{figure}


\section{Background and Definitions}
\label{sec:background}

We work with densities defined on a subset $\mathcal{X} \subset \mathbb{R}^d$, and denote
by $\|. \|$ the
Euclidean norm on $\mathcal{X}$.
Throughout this paper we restrict our attention to cluster tree estimators that are specified in 
terms of a function $f: \mathcal{X} \mapsto [0,\infty)$, i.e. we have the following definition:
\begin{definition}
	\label{def:tree}
	For any $f: \mathcal{X} \mapsto [0,\infty)$
	the {\em cluster tree} of $f$ is a function 
	$\treef: \mathbb{R} \mapsto 2^{\mathcal{X}}$, where
	$2^{\mathcal{X}}$ is the set of all subsets of $\mathcal{X}$,
	and $T_f(\lambda)$ is the set of the connected components
	of the upper-level set $\{x \in \mathcal{X}: f(x) \geq \lambda\}$.
	We define 
	the collection of connected components $\colltf$, 
	as $\colltf=\underset{\lambda}{\bigcup}~\treef(\lambda)$.
\end{definition}

As will be clearer in what follows, working only with cluster trees defined
via a function $f$ simplifies our search for metrics on trees, allowing
us to use metrics specified in terms of the function $f$. With a slight
abuse of notation, we will use $\treef$ to denote also $\colltf$, and write $\cluster \in \treef$ to signify $\cluster \in \colltf$.
The cluster tree $\treef$ indeed has a tree structure,
since for every pair $C_{1},C_{2}\in \treef$, either 
$C_{1}\subset C_{2}$, $C_{2}\subset C_{1}$, or $C_{1}\cap C_{2}=\emptyset$ holds.
See Figure~\ref{fig:ex} for
a graphical illustration of
a cluster tree.
The formal definition of the tree requires some topological theory; these details are in Appendix \ref{app:partial}.

In the context of hierarchical clustering, we are often interested in the ``height'' at
which two points or two clusters merge in the clustering. We introduce the merge height from  \cite[Definition 6]{eldridge2015beyond}:
\begin{definition}
	\label{def:mergeheight}
	For any two points $x,y\in\mathcal{X}$, any $f: \mathcal{X} \mapsto [0,\infty)$, and its tree $\treef$, their {\bf merge height} $m_{f}(x,y)$ is defined as the largest $\lambda$ such that $x$ and $y$ are in the same density cluster at level $\lambda$, i.e. 
	\[
	m_{f}(x,y) = \sup\left\{ \lambda\in\mathbb{R}:\,\text{there exists }C\in T_{f}(\lambda)\text{ such that }x,y\in C\right\} .
	\]
	We refer to the function $m_{f}: \mathcal{X} \times \mathcal{X} \mapsto \mathbb{R}$
	as the merge height function.
	For any two clusters $C_{1},C_{2}\in\colltf$, their merge height $m_{f}(C_{1},C_{2})$ is defined analogously, 
	\[
	m_{f}(C_{1},C_{2}) = \sup\left\{ \lambda\in\mathbb{R}:\,\text{there exists }C\in T_{f}(\lambda)\text{ such that }C_{1},C_{2}\subset C\right\} .
	\]
\end{definition}

One of the contributions of this paper is to construct valid confidence sets
for the unknown true tree and to develop methods for visualizing the trees
contained in this confidence set. Formally,
we assume that we have samples $\{X_1,\ldots,X_n\}$ from a distribution
$\mPstar$
with density $\densstar$.
\begin{definition}
	\label{def:valid}
	An asymptotic $(1 - \alpha)$ confidence set, $C_\alpha$, is a collection of trees with the
	property that
	\begin{align*}
	\mPstar(\treepstar \in C_{\alpha}) = 1 - \alpha + o(1).
	\end{align*}
\end{definition}
We also provide non-asymptotic upper bounds on the $o(1)$ term in the above definition. 
Additionally, we provide methods to summarize the confidence set above. In order 
to summarize the confidence set, we define a partial ordering on trees. 
\begin{definition}
	\label{def:partial}
	For any $f,g: \mathcal{X} \mapsto [0,\infty)$ and their trees $\treef$, $\treeg$, we say $\treef \preceq \treeg$ if there exists a map 
	$\Phi:\colltf \rightarrow \colltg$ such that for any $C_1,C_2\in \treef$, we have $C_{1}\subset C_{2}$ if and only if $\Phi(C_{1})\subset\Phi(C_{2})$.
\end{definition}
With Definition \ref{def:valid} and \ref{def:partial}, we describe the confidence set succinctly via some 
of the \emph{simplest} trees in the confidence set in Section \ref{sec:confidence_set}. Intuitively, these are trees without statistically insignificant splits.

It is easy to check that the partial order $\preceq$ in Definition \ref{def:partial} is reflexive 
(i.e. $T_{f}\preceq T_{f}$) and transitive (i.e. that $T_{f_{1}}\preceq T_{f_{2}}$ and $T_{f_{2}}\preceq T_{f_{3}}$ implies $T_{f_{1}}\preceq T_{f_{3}}$). 
However, 
to argue that $\preceq$ is a partial order, we need to show the antisymmetry, 
i.e. $T_{f}\preceq T_{g}$ and $T_{g}\preceq T_{f}$ implies that $T_{f}$ and $T_{g}$ are equivalent in some sense. 
In Appendices \ref{app:topology} and \ref{app:partial}, we show an important result:
for an appropriate topology on trees,
$T_{f}\preceq T_{g}$ and $T_{g}\preceq T_{f}$ implies that $\treef$ and $\treef$ are \emph{topologically equivalent}.

\begin{figure}
	\centering
	\begin{subfigure}{0.32\linewidth}
		\centering
		\includegraphics[scale=0.3]{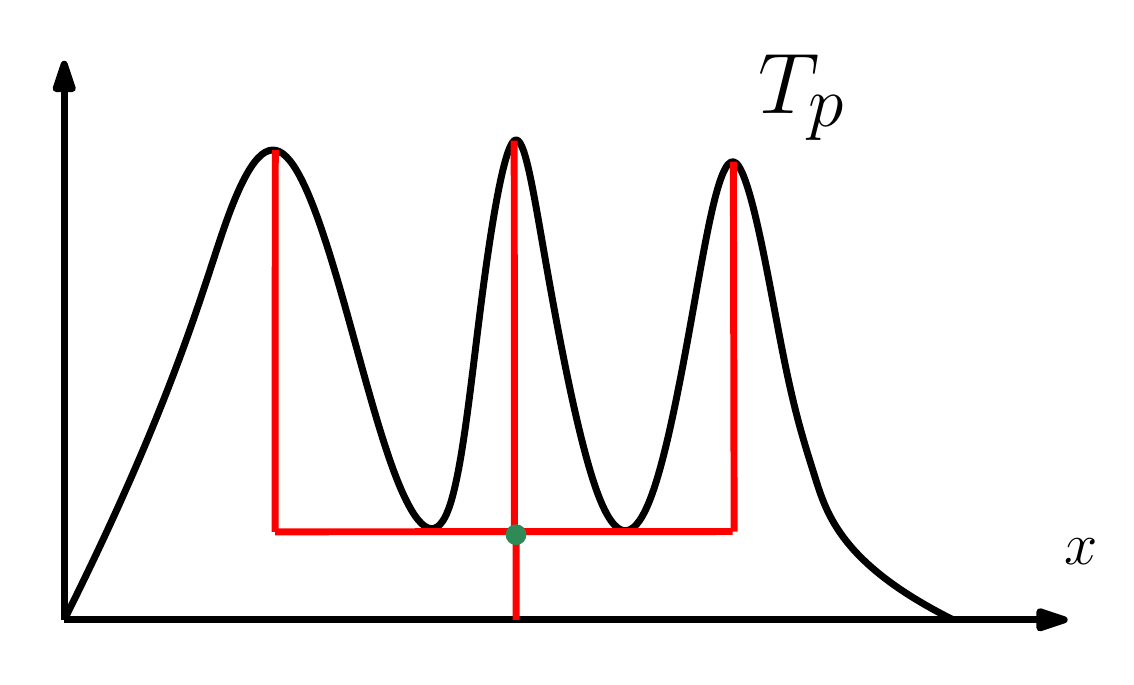}
		\caption{\label{subfig:partial_lessedge}}
	\end{subfigure}
	\begin{subfigure}{0.32\linewidth}			
		\centering
		\includegraphics[scale=0.3]{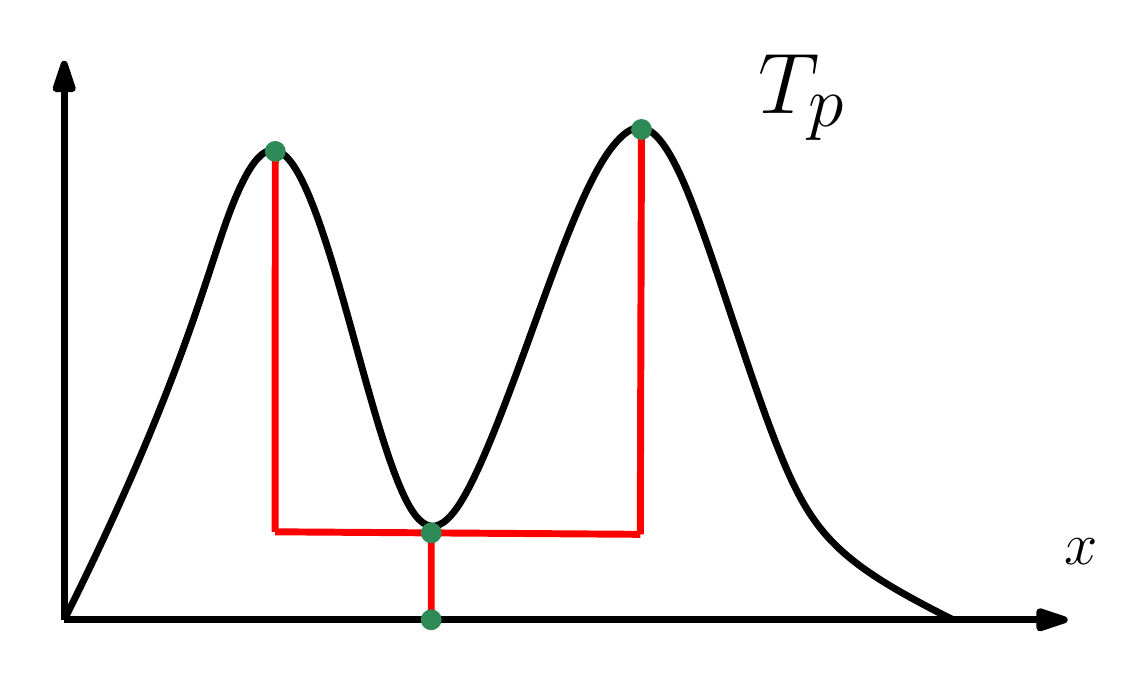}
		\caption{\label{subfig:partial_subtractedge}}
	\end{subfigure}
	\begin{subfigure}{0.32\linewidth}
		\centering
		\includegraphics[scale=0.3]{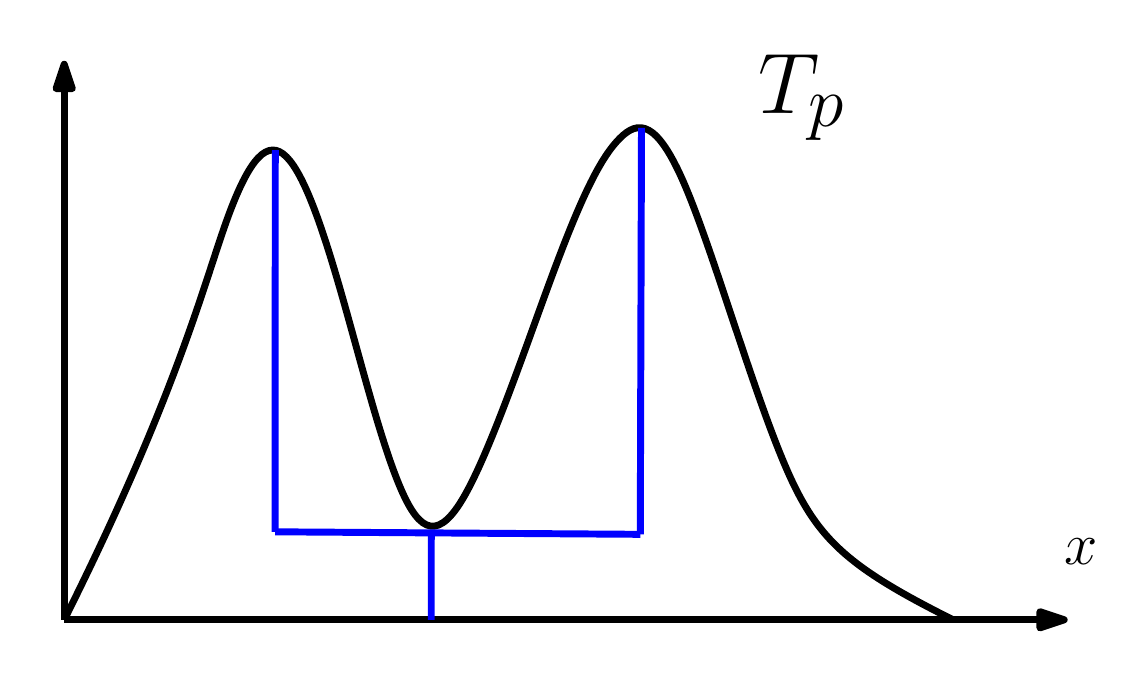}
		\caption{\label{subfig:partial_embeddedtree}}
	\end{subfigure}
	\\
	\begin{subfigure}{0.32\linewidth}
		\centering
		\includegraphics[scale=0.3]{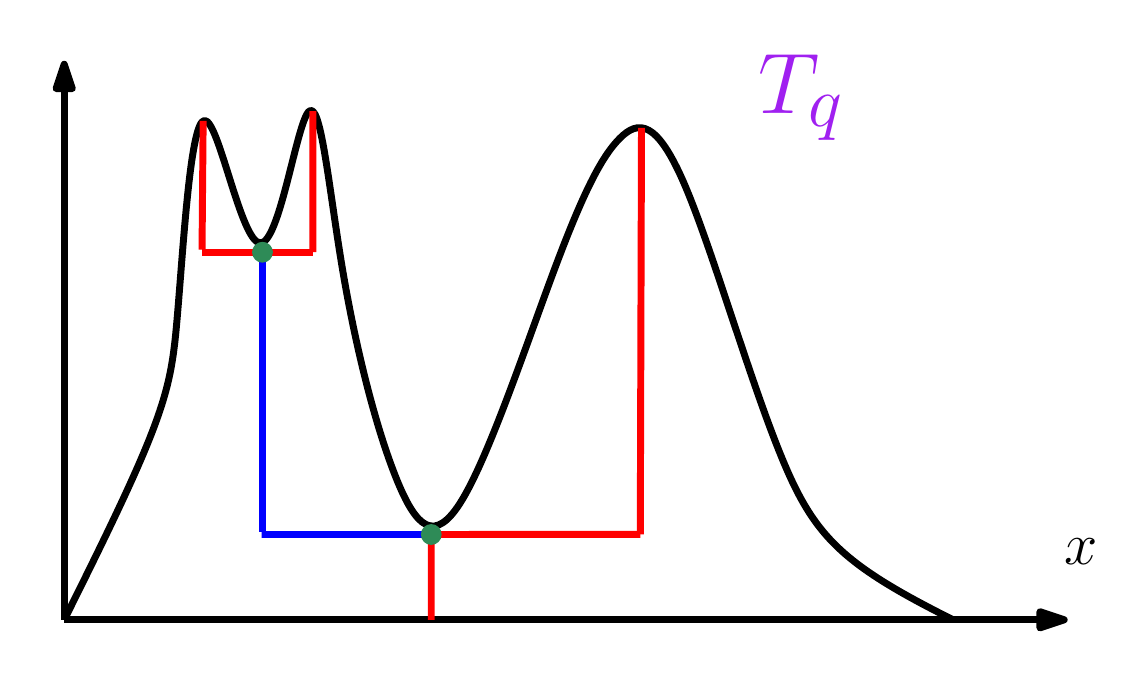}
		\caption{\label{subfig:partial_moreedge}}
	\end{subfigure}
	\begin{subfigure}{0.32\linewidth}
		\centering
		\includegraphics[scale=0.3]{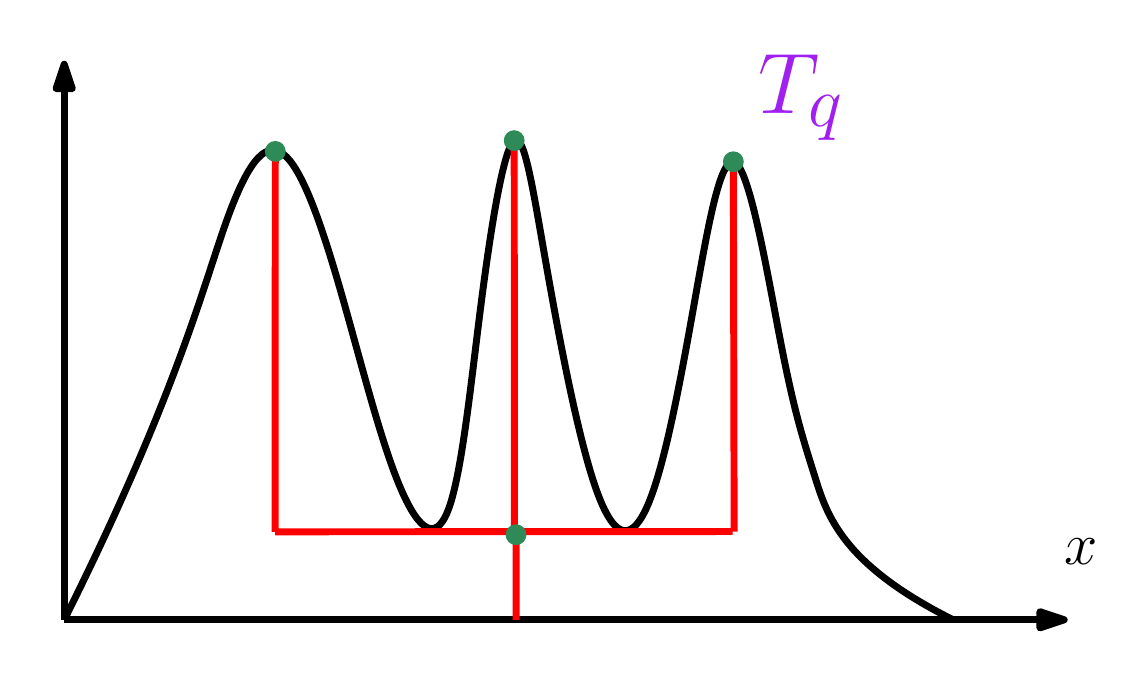}
		\caption{\label{subfig:partial_addedge}}
	\end{subfigure}
	\begin{subfigure}{0.32\linewidth}	
		\centering
		\includegraphics[scale=0.3]{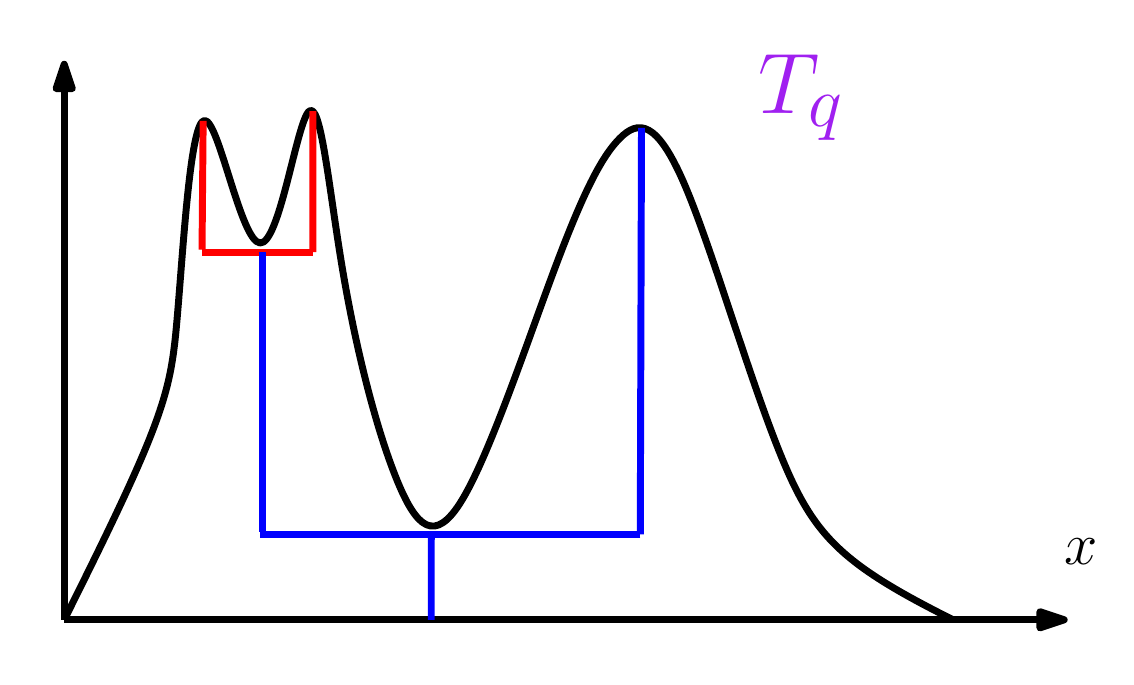}
		\caption{\label{subfig:partial_embeddingtree}}
	\end{subfigure}
	\caption{Three illustrations of the partial order $\preceq$ in Definition \ref{def:partial}.
		In each case, in agreement with our intuitive notion of simplicity, the tree on the top (\protect\subref{subfig:partial_lessedge}, \protect\subref{subfig:partial_subtractedge}, and \protect\subref{subfig:partial_embeddedtree}) is lower than the corresponding tree on the bottom(\protect\subref{subfig:partial_moreedge},  \protect\subref{subfig:partial_addedge}, and \protect\subref{subfig:partial_embeddingtree}) in the partial order, i.e. for each example $\treep \preceq \treeq$.
	}
	\label{fig:partial}
\end{figure}

The partial order $\preceq$ in
Definition \ref{def:partial} matches intuitive notions
of the complexity of the tree for several reasons (see Figure~\ref{fig:partial}).
Firstly, $T_{f}\preceq T_{g}$ implies 
$\text{(number of edges of }T_{f})\leq\text{(number of edges of }T_{g})$ (compare Figure~\ref{fig:partial}\subref{subfig:partial_lessedge} and \subref{subfig:partial_moreedge}, and see Lemma~\ref{lem:partial_edgenumber} in Appendix~\ref{app:partial}).
Secondly, if $T_{g}$ is obtained from $T_{f}$ by adding edges, then $T_{f}\preceq T_{g}$
(compare Figure~\ref{fig:partial}\subref{subfig:partial_subtractedge} and \subref{subfig:partial_addedge}, and see Lemma~\ref{lem:partial_insert} in Appendix~\ref{app:partial}).
Finally, the existence of a topology preserving embedding from 
$\colltf$ to $\colltg$
implies the relationship $\treef \preceq \treeg$ (compare Figure~\ref{fig:partial}\subref{subfig:partial_embeddedtree} and \subref{subfig:partial_embeddingtree}, and see Lemma~\ref{lem:partial_embedding} in Appendix~\ref{app:partial}).

\section{Tree Metrics}	
\label{sec:metric}
In this section, we introduce some natural metrics 
on cluster trees and study some of their properties that determine their
suitability for statistical inference. 
We let $p,q:\mathcal{X}\to[0,\infty)$ be nonnegative functions and let
$T_p$ and $T_q$ be the corresponding trees. 

\subsection{Metrics}
We consider three metrics on cluster trees, the first is the 
standard $\ellinf$ metric, while the second and third are metrics
that appear in the work of Eldridge et al. \cite{eldridge2015beyond}.

{\bf $\ellinf$ metric: } The simplest metric is
$d_\infty(T_p,T_q) = \|p-q\|_\infty = \sup_{x \in \mathcal{X}} | p(x) - q(x) |.$
We will show in what follows
that, in the context of statistical inference, this metric has several advantages
over other metrics.

{\bf Merge distortion metric: } The merge distortion metric intuitively 
measures the discrepancy in the merge height functions of
two trees in Definition \ref{def:mergeheight}. We consider the
{\em merge distortion metric} \cite[Definition 11]{eldridge2015beyond} defined by
$$
\dm(T_p,T_q) =\sup_{x,y \in \mathcal{X}} |m_p(x,y) - m_q(x,y)|.
$$
The merge distortion metric we consider is a special case of the metric introduced 
by \citet{eldridge2015beyond}\footnote{They further allow flexibility in 
	taking a $\sup$ over a subset of $\mathcal{X}$.}. 
The merge distortion metric was introduced by \citet{eldridge2015beyond} to study
the convergence of cluster tree estimators. They establish several interesting 
properties of the merge distortion metric: in particular, the metric is stable to 
perturbations in $\ellinf$, and further, that convergence in the merge distortion metric
strengthens previous notions of convergence of the cluster trees. 


{\bf Modified merge distortion metric: } 
We also consider the {\em modified
	merge distortion metric} given by
$$
\dmm(T_p,T_q) =\sup_{x,y \in \mathcal{X}} |d_{T_{p}}(x,y) - d_{T_{q}}(x,y)|,
$$
where $d_{T_{p}}(x,y) = p(x)+p(y) -2 m_p(x,y),$ which corresponds
to the (pseudo)-distance between $x$ and $y$ \emph{along} the tree. 
The metric $\dmm$ is used in various proofs in the work of \citet{eldridge2015beyond}. 
It is sensitive to both distortions of the merge heights in Definition \ref{def:mergeheight}, as well as of the underlying densities. 
Since the metric captures the distortion of distances between points along the tree, it is in some sense most closely aligned with the cluster tree. 
Finally, it is worth noting that unlike the interleaving distance and the functional distortion metric~\cite{bauer2015strong,morozov2013interleaving}, the three metrics we consider in this
paper are quite simple to approximate to a high-precision. 

\subsection{Properties of the Metrics}
The following Lemma gives some basic relationships between the three
metrics $d_\infty, \dm$ and $\dmm$. We 
define $p_{\inf}=\inf_{x \in \mathcal{X}} p(x)$, and $q_{\inf}$
analogously, and 
$a = \inf_{x \in \mathcal{X}} \{p(x) + q(x)\} - 2\min \{ p_{\inf},\,q_{\inf} \}$.
Note that when the Lebesgue measure $\mu(\mathcal{X})$ is infinite, then $p_{\inf} = q_{\inf} = a = 0$.

\begin{lemma}
	\label{lem:infty_M}
	For any densities $p$ and $q$, the following relationships hold:
	(i) When $p$ and $q$ are continuous, then $d_\infty(T_p,T_q) = \dm(T_p,T_q).$
	(ii) $\dmm(T_{p},T_{q})\leq4d_{\infty}(T_{p},T_{q}).$
	(iii) $\dmm(T_{p},T_{q})\geq d_{\infty}(T_{p},T_{q})-a$, where $a$ is defined as above. Additionally when  
	$\mu(\mathcal{X}) = \infty$, then $\dmm(T_{p},T_{q})\geq d_{\infty}(T_{p},T_{q})$.
\end{lemma}
The proof is in Appendix \ref{app:proof_metric_hadamard}.
From Lemma \ref{lem:infty_M}, we can see that under a mild assumption (continuity of the densities), 
$d_\infty$ and $\dm$ are equivalent. We note again that the work of \citet{eldridge2015beyond}
actually defines a family of merge distortion metrics, while we restrict our attention to a canonical one.
We can also see from Lemma \ref{lem:infty_M} that while the modified merge metric is not equivalent to $d_\infty$,
it is usually  multiplicatively sandwiched by $d_\infty$.



Our next line of investigation is aimed at assessing the suitability of the three metrics
for the task of statistical inference. Given the strong 
equivalence of $d_\infty$ and $\dm$ we focus our attention on $d_\infty$ and $\dmm$. 
Based on prior work (see \cite{chen2015density,cck}), the large sample behavior 
of $d_\infty$ is well understood. In particular, $d_\infty(\esttree,\treepstar)$ 
converges to the supremum of an appropriate Gaussian process, on the basis
of which we can construct confidence intervals for the $d_\infty$ metric.

The situation for the metric $\dmm$ is substantially more subtle.
One of our eventual goals is to use the non-parametric bootstrap to construct
valid estimates of the confidence set. 
In general, a way to assess the amenability 
of a functional to the bootstrap is via \emph{Hadamard differentiability} \cite{wellner2013weak}.
Roughly speaking, Hadamard-differentiability
is a type of {\em statistical stability}, that ensures that the functional
under consideration is stable to perturbations in the input distribution.
In Appendix~\ref{app:hadamard}, we formally define Hadamard differentiability
and prove that $\dmm$ is \emph{not} point-wise Hadamard differentiable. 
This does not completely rule out the possibility of finding a way to construct confidence sets based on $\dmm$,
but doing so would be difficult and so far we know of no way to do it.

In summary, based on computational considerations we eliminate the 
interleaving distance and the functional distortion 
metric~\cite{bauer2015strong,morozov2013interleaving}, we eliminate the $\dmm$
metric based on its unsuitability for statistical inference and focus the 
rest of our paper on the $d_\infty$ (or equivalently $\dm$) metric
which is both computationally tractable and has well understood 
statistical behavior.

\section{Confidence Sets}	
\label{sec:confidence_set}

In this section, we consider the construction of valid confidence intervals centered
around the kernel density estimator, defined in Equation~\eqref{eq:kde}.
%
We first observe that a fixed bandwidth for the KDE gives a dimension-free rate of convergence for estimating a cluster tree. For estimating a density in high dimensions, the KDE has a poor rate of convergence, due to a decreasing bandwidth for simultaneously optimizing the bias and the variance of the KDE.

When estimating a cluster tree, the bias of the KDE does not affect its cluster tree. Intuitively, the cluster tree is a shape characteristic of a function, which is not affected by the bias. 
Defining the \emph{biased} density, 
$p_h(x) = \mathbb{E} [\estdens(x)]$,
two cluster trees from $p_{h}$ and the true density $\densstar$ are equivalent
with respect to the topology in Appendix \ref{app:topology}, if $h$ is small
enough and $\densstar$ is regular enough:
\begin{lemma}
	\label{lem:morse}
	Suppose that the true unknown density $\densstar$, has no non-degenerate critical points
	\footnote{The
		Hessian of $\densstar$ at every critical point is non-degenerate. Such functions are known 
		as Morse functions.},
	then there exists a constant 
	$h_0 > 0$ such that for all
	$0 < h \leq h_0$, the two cluster trees, $\treepstar$ and $\treeph$ have the same topology in Appendix \ref{app:topology}.
\end{lemma}
From Lemma~\ref{lem:morse}, proved in Appendix~\ref{app:proof_confidence_set}, a fixed bandwidth for the KDE can be applied to give a dimension-free rate of convergence for estimating the cluster tree. Instead of decreasing bandwidth $h$ and inferring the cluster tree of the true density $\treepstar$ at rate $O_P(n^{-2/(4+d)})$, Lemma~\ref{lem:morse} implies that we can fix $h>0$ and infer the cluster tree of the biased density $\treeph$ at rate $O_P(n^{-1/2})$ \emph{independently of the dimension}. Hence a fixed bandwidth crucially enhances the convergence rate of the proposed methods in high-dimensional settings.

\subsection{A data-driven confidence set}

We recall that we base our inference on the $d_\infty$ metric, and we recall the definition 
of a valid confidence set (see Definition~\ref{def:valid}).
As a conceptual first step, suppose that for a specified value $\alpha$
we could compute the $1 - \alpha$ quantile of the distribution of 
$d_\infty(\esttree,\treeph)$, and denote this value $t_\alpha$. 
Then a valid confidence set for the unknown $\treeph$ is 
$C_\alpha = \{T: d_\infty(T,\esttree) \leq t_\alpha\}$.
To estimate $t_\alpha$,
we use the bootstrap.
Specifically, we generate $B$ bootstrap samples, $\{\tilde{X}_{1}^1,\cdots,\tilde{X}_{n}^1\}, \ldots,
\{\tilde{X}_{1}^B,\cdots,\tilde{X}_{n}^B\},$
by sampling with replacement from
the original sample. On each bootstrap sample, we compute the 
KDE, and the associated cluster tree. We denote the cluster
trees $\{\tilde{T}_{p_h}^1,\ldots,\tilde{T}_{p_h}^B\}$.
Finally, we estimate $t_\alpha$ by
\[
\hat{t}_\alpha =\hat{F}^{-1}(1-\alpha), \ \mbox{where} \quad \hat{F}(s) = \frac{1}{B}\sum_{i=1}^n \mathbb{I}(d_\infty(
\tilde{T}_{p_h}^i, \esttree)<s).
\]
Then the data-driven confidence set is
$\hat{C}_\alpha = \{T: d_\infty(T, \hat{T}_h)\leq \hat{t}_\alpha\}.$
Using techniques from \cite{cck,chen2015density},
the following can be shown (proof omitted):
\begin{theorem} 
	\label{thm:CI}
	Under mild regularity conditions on
	the kernel\footnote{See Appendix~\ref{app:regularity} for details.}, 
	we have that the constructed confidence set is asymptotically valid and satisfies,
	\[
	\mathbb{P}\left(T_h \in \hat{C}_\alpha\right) = 1-\alpha + O\Big(\Big(\frac{\log^7 n}{nh^{d}}\Big)^{1/6}\Big).
	\]
\end{theorem}
Hence our data-driven confidence set is consistent at dimension independent rate. When $h$ is a fixed small constant, Lemma \ref{lem:morse} implies that $\treepstar$ and $\treeph$ have the same topology, and Theorem \ref{thm:CI} guarantees
that the non-parametric bootstrap is consistent at a dimension independent $O(((\log n)^7/n)^{1/6})$
rate. For reasons explained in \cite{cck}, this rate is believed to be optimal.

\subsection{Probing the Confidence Set}
\label{subsec:pruning}

The confidence set $\hat{C}_{\alpha}$ 
is an infinite set with a complex structure. 
Infinitesimal perturbations of the density estimate are in our confidence set and so this set
contains very complex trees.
One way to understand the structure of the confidence set is to focus attention on
simple trees in the confidence set. Intuitively, these trees only contain topological features 
(splits and branches) that are sufficiently strongly supported by the data.

We propose two \emph{pruning} schemes to find trees, that are simpler than the empirical 
tree $\esttree$ 
that are in the confidence set. Pruning the empirical tree aids visualization
as well as de-noises the empirical tree by eliminating some 
features that arise solely due to the stochastic variability of the finite-sample.
The algorithms are (see Figure~\ref{fig:prune}):\\
1. {\bf Pruning only leaves:} Remove all leaves of length less than $2\hat t_{\alpha}$  (Figure~\ref{fig:prune}\subref{subfig:prune_leaf}).\\
2. {\bf Pruning leaves and internal branches:} In this case, we first prune the leaves as above.
This yields a new tree. Now we again prune (using cumulative length) any leaf of length less than $2\hat t_{\alpha}$.
We continue iteratively until all remaining leaves are of cumulative length larger than $2\hat t_{\alpha}$ (Figure~\ref{fig:prune}\subref{subfig:prune_branch}).

In Appendix \ref{app:pruning} we formally define the pruning operation and
show the following.
The remaining tree $\tilde T$ after either of the above pruning operations satisfies:
(i) $\tilde T \preceq \esttree$, (ii)
there exists a function $f$ whose tree is $\tilde T$, and
(iii) $\tilde T\in \hat C_\alpha$ (see Lemma \ref{lem:pruning} in Appendix \ref{app:pruning}).
In other words,
we identified
a valid tree with a statistical guarantee that is simpler than the original estimate $\esttree$.
Intuitively, some of the statistically insignificant features have been removed from
$\esttree$.
We should point out, however, that there may exist other trees
that are simpler than $\esttree$ that are in
$\hat C_\alpha$.
Ideally, we would like to have an algorithm that identifies all trees in
the confidence set that are minimal with respect to the partial order $\preceq$ in Definition \ref{def:partial}.
This is an open question that we will address in future work.

\begin{figure}
	\centering
	\begin{subfigure}{0.32\linewidth}
		\centering
		\includegraphics[scale=0.5]{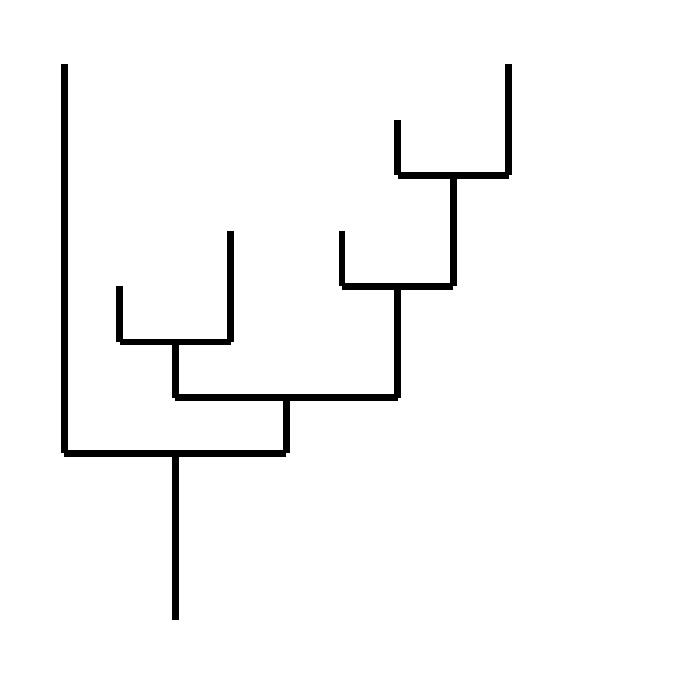}
		\caption{\label{subfig:prune_empirical} The empirical tree.}
	\end{subfigure}
	\begin{subfigure}{0.32\linewidth}	
		\centering
		\includegraphics[scale=0.5]{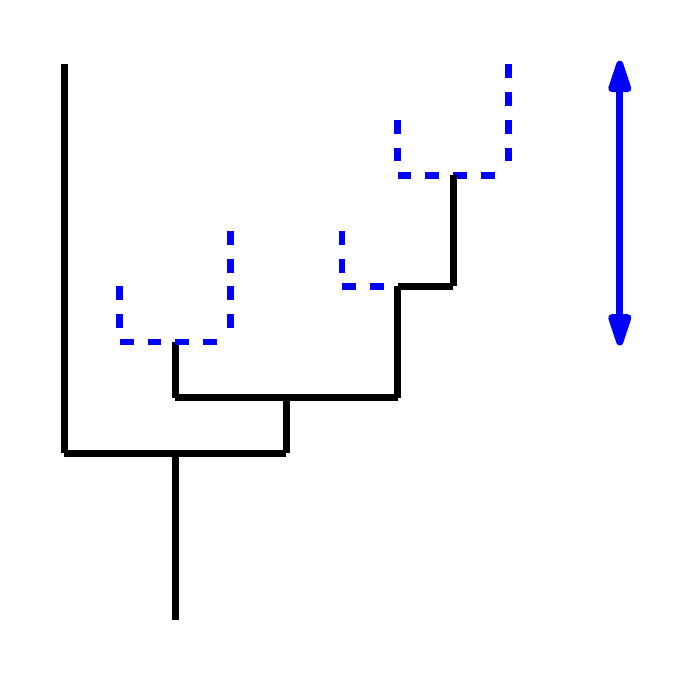}
		\caption{\label{subfig:prune_leaf} Pruning only leaves.}
	\end{subfigure}
	\begin{subfigure}{0.32\linewidth}	
		\centering
		\includegraphics[scale=0.5]{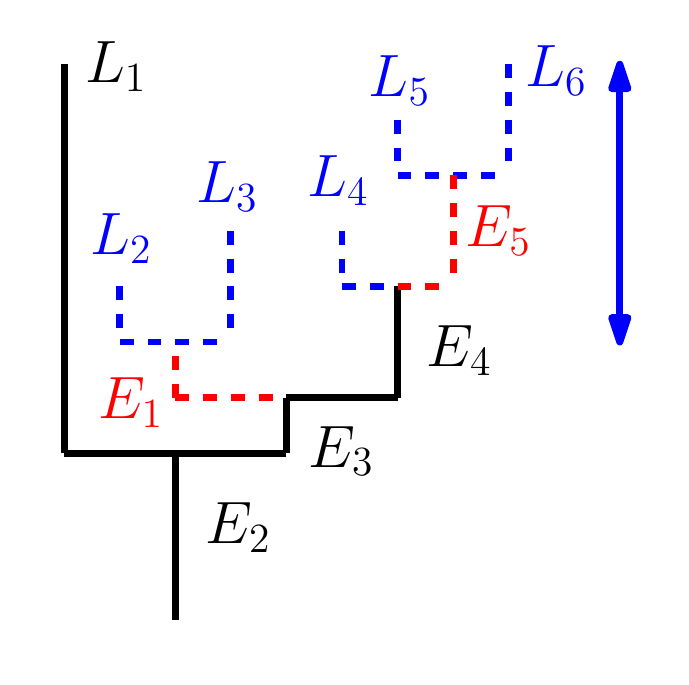}
		\caption{\label{subfig:prune_branch} Pruning leaves and branches.}
	\end{subfigure}
	\caption{Illustrations of our two pruning strategies. \subref{subfig:prune_empirical} shows the empirical
		tree. In \subref{subfig:prune_leaf}, leaves that are insignificant are pruned, while in \subref{subfig:prune_branch}, insignificant internal branches are further pruned top-down.}
	\label{fig:prune}
\end{figure}

\begin{figure}
	\centering
	\begin{subfigure}{0.32\linewidth}	
		\centering
		\includegraphics[width=4cm]{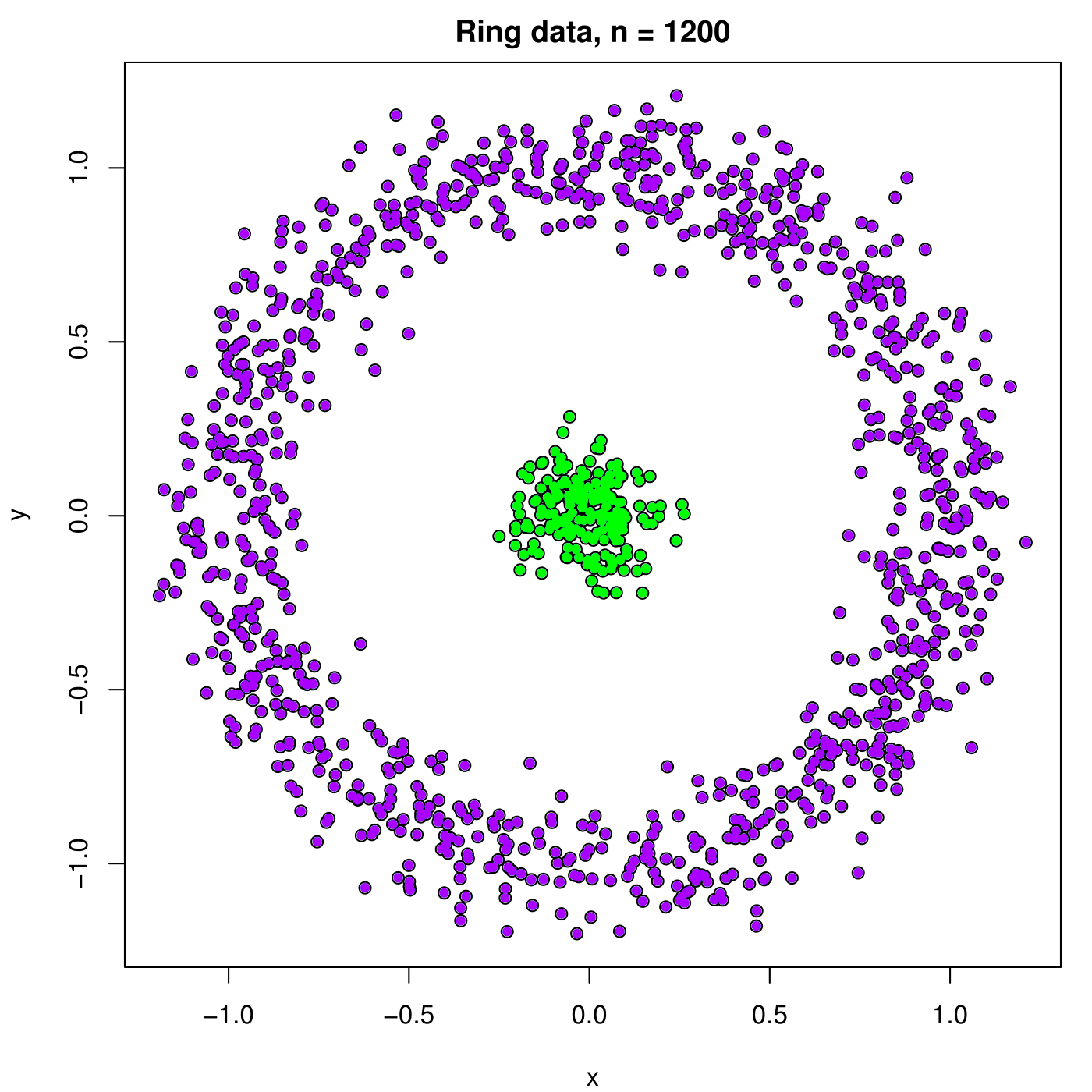}
		\caption{\label{subfig:sim_data_ring}}
	\end{subfigure}
	\begin{subfigure}{0.32\linewidth}	
		\centering
		\includegraphics[width=4cm]{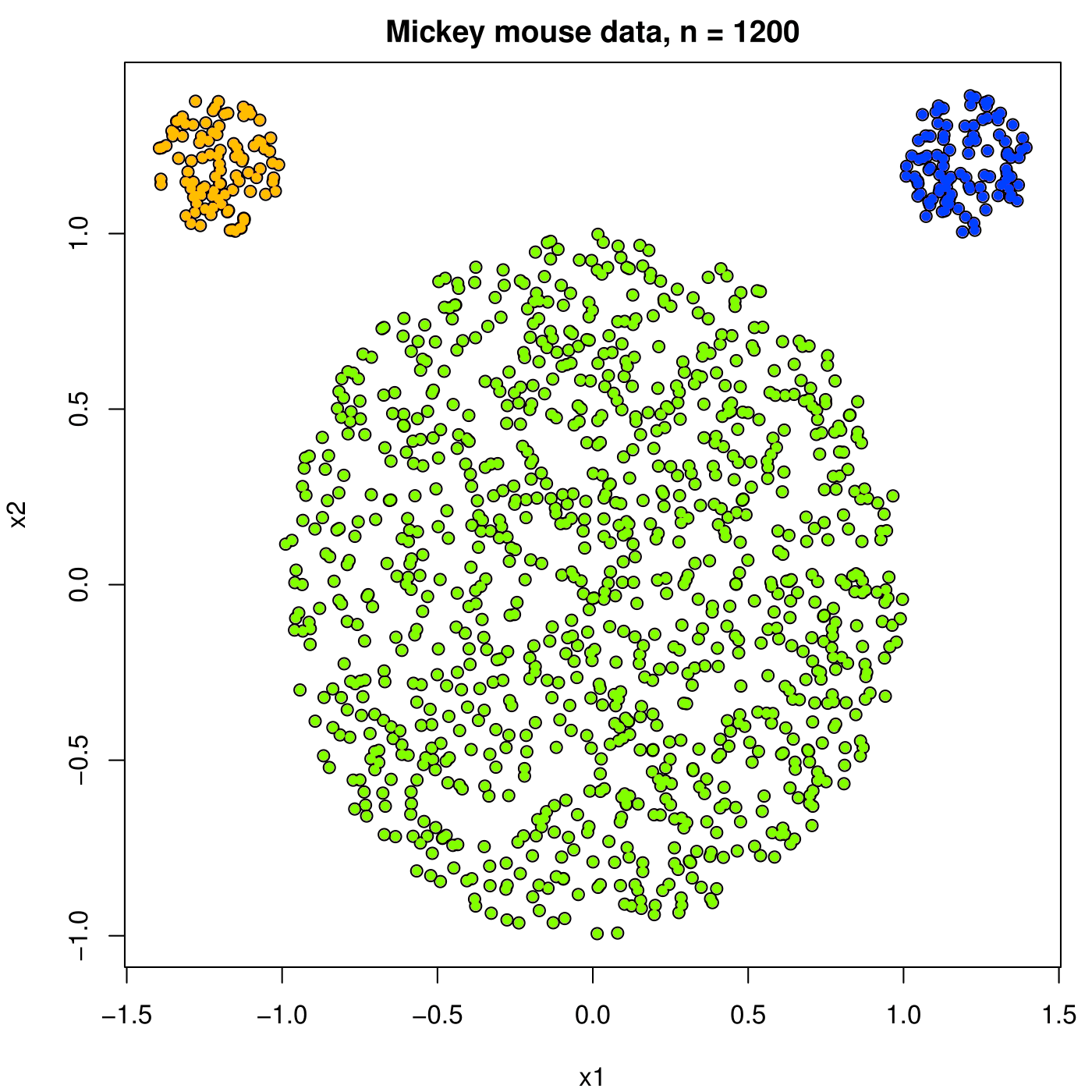}
		\caption{\label{subfig:sim_data_mouse}}
	\end{subfigure}
	\begin{subfigure}{0.32\linewidth}	
		\centering
		\includegraphics[width=4cm]{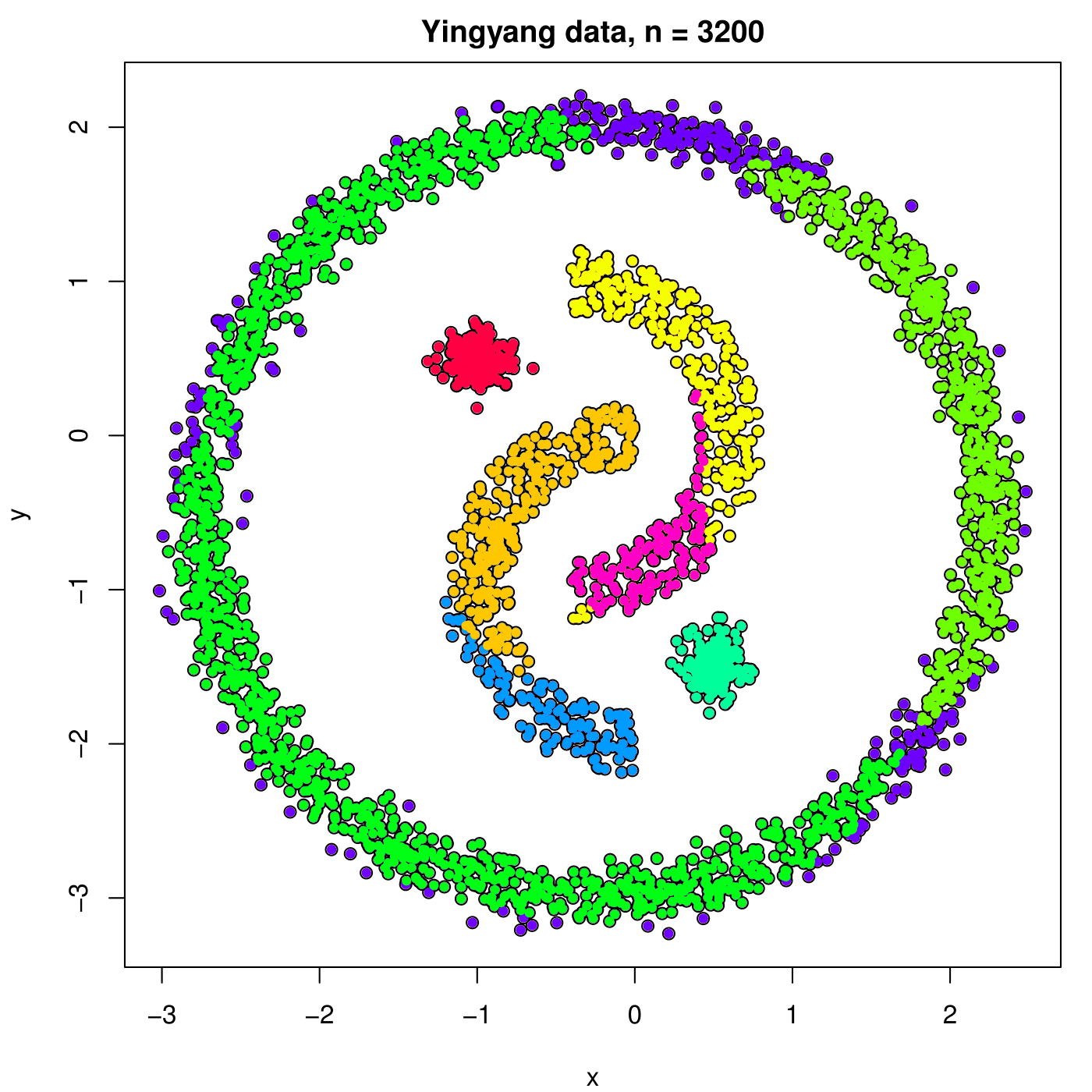}
		\caption{\label{subfig:sim_data_yingyang}}
	\end{subfigure}
	\\
	\begin{subfigure}{0.32\linewidth}	
		\centering
		\includegraphics[width=4cm]{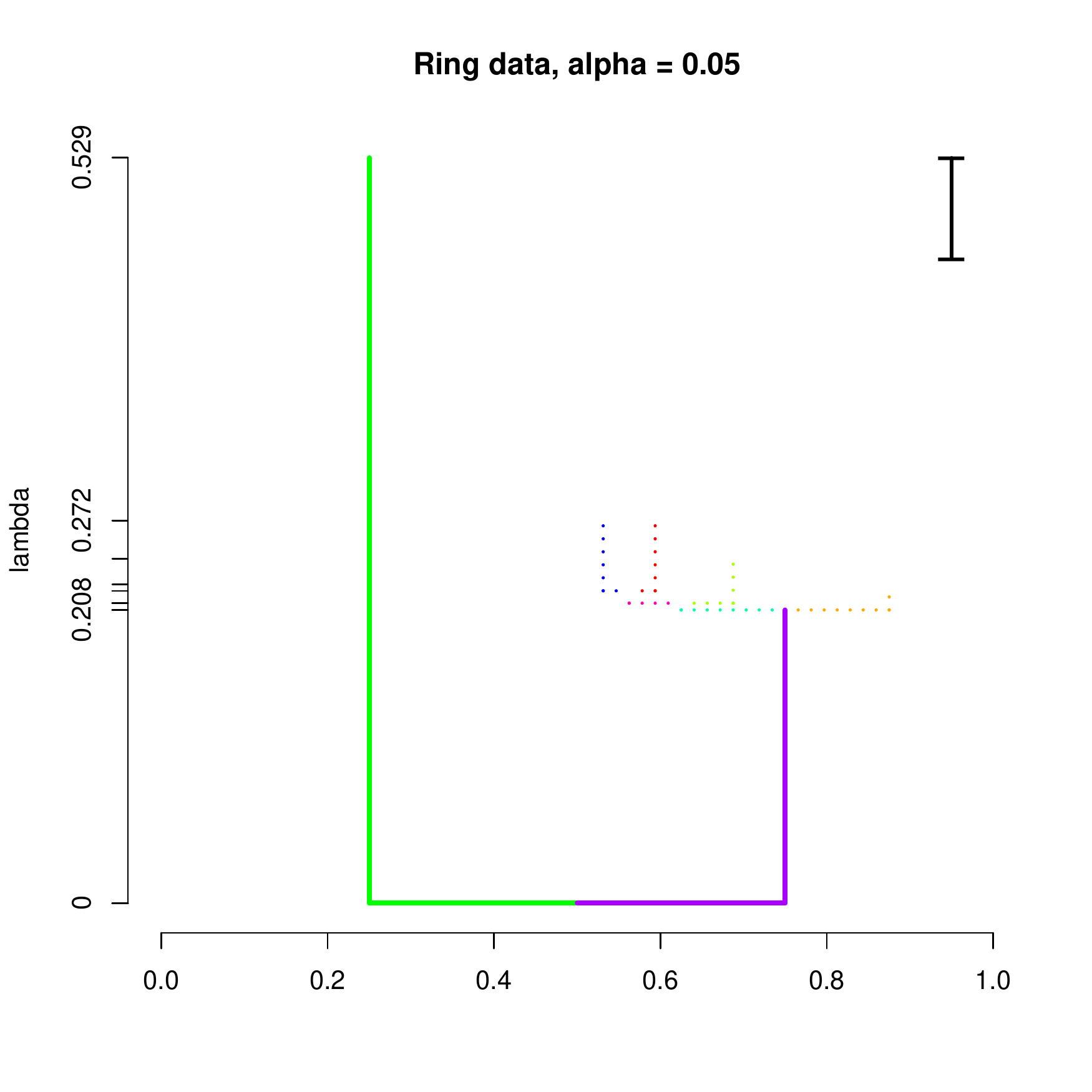}
		\caption{\label{subfig:sim_tree_ring}}
	\end{subfigure}
	\begin{subfigure}{0.32\linewidth}	
		\centering
		\includegraphics[width=4cm]{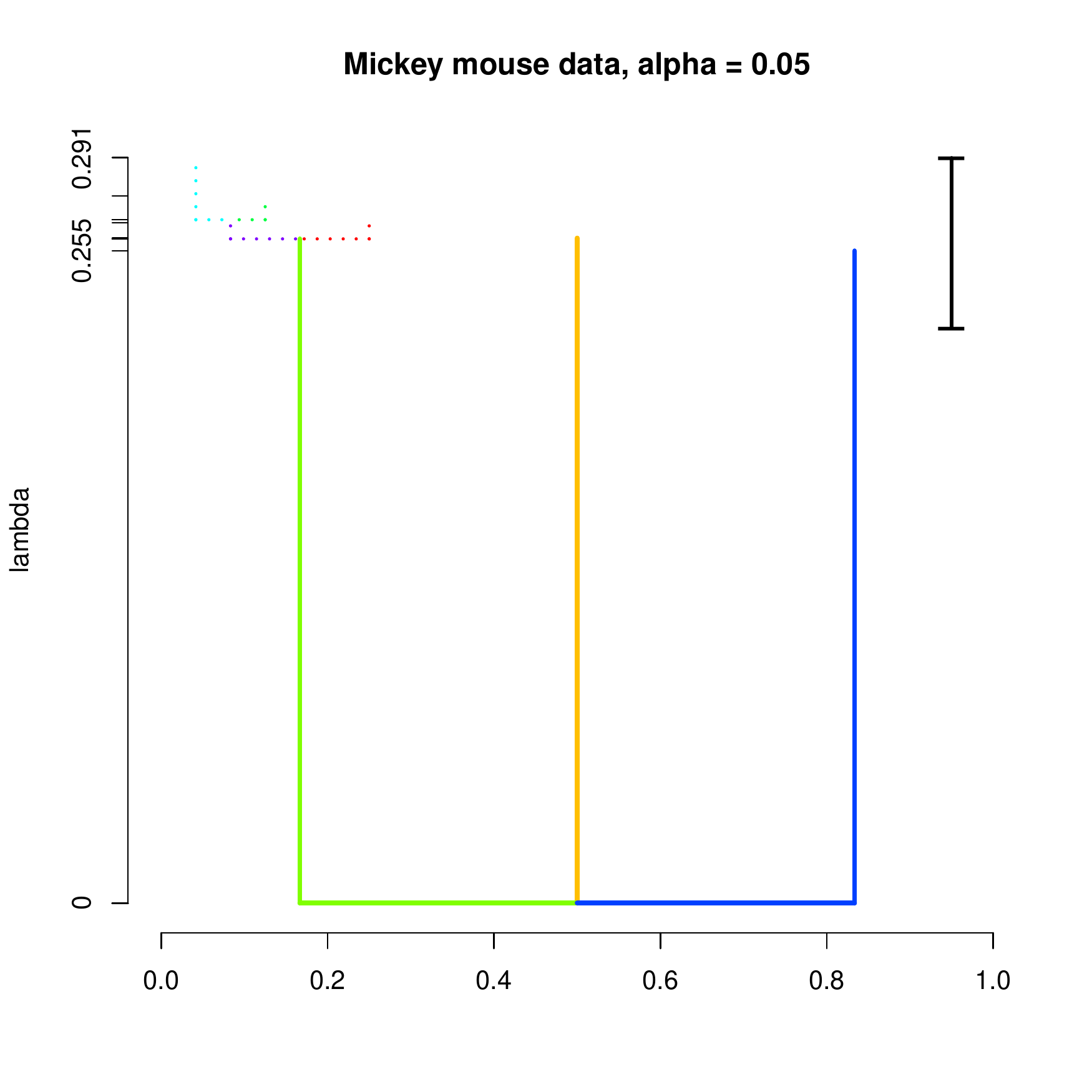}
		\caption{\label{subfig:sim_tree_mouse}}
	\end{subfigure}
	\begin{subfigure}{0.32\linewidth}	
		\centering
		\includegraphics[width=4cm]{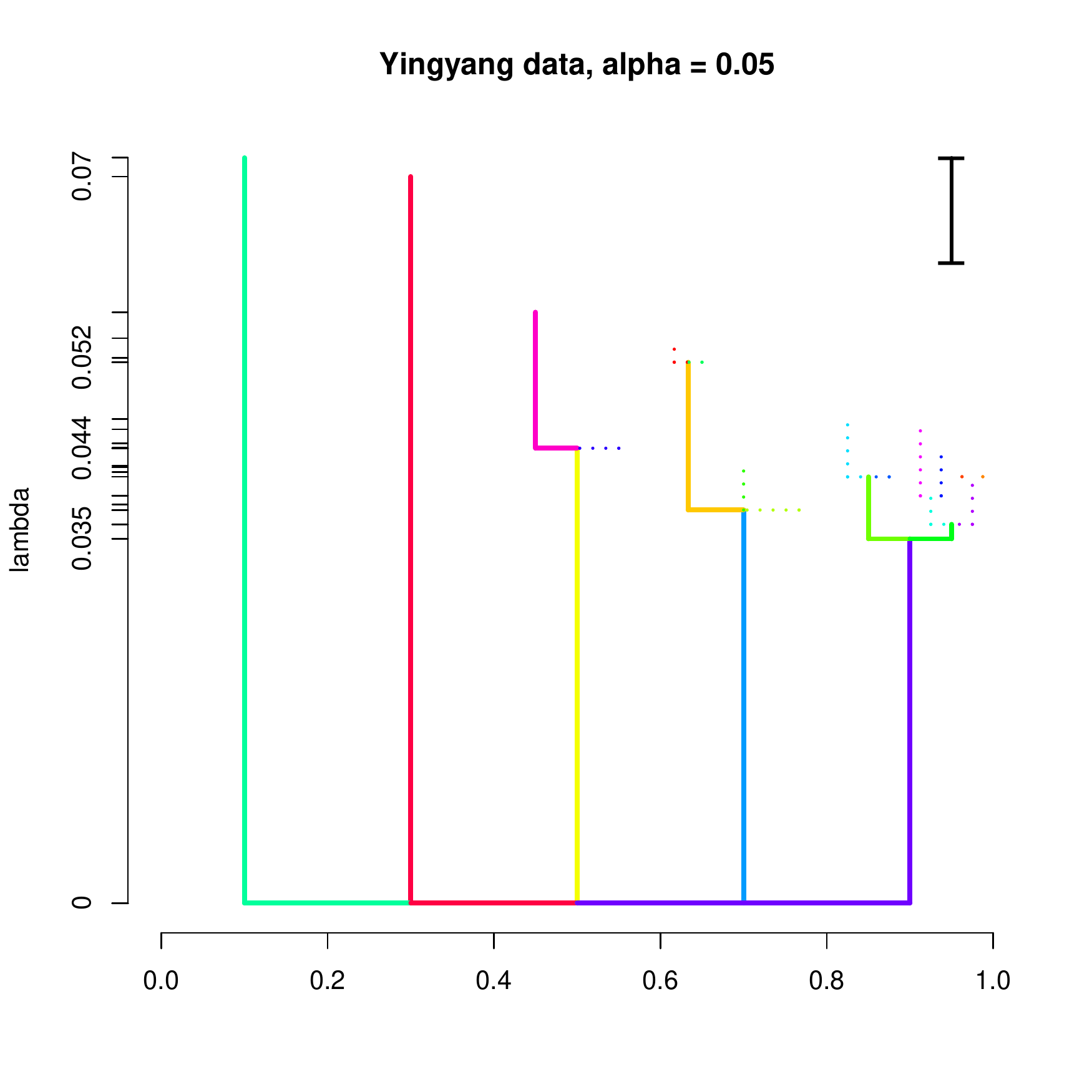}
		\caption{\label{subfig:sim_tree_yingyang}}
	\end{subfigure}
	\caption{Simulation examples.
		\subref{subfig:sim_data_ring} and \subref{subfig:sim_tree_ring} are the ring data; \subref{subfig:sim_data_mouse} and \subref{subfig:sim_tree_mouse} are the mickey mouse data;
		\subref{subfig:sim_data_yingyang} and \subref{subfig:sim_tree_yingyang} are the yingyang data. 
		The solid lines are the pruned trees; 
		the dashed lines are leaves (and edges) removed
		by the pruning procedure.
		A bar of length $2\hat{t}_\alpha$ is at the top right corner.
		The pruned trees recover the actual structure of connected components.
	}
	\label{fig:sim}
\end{figure}

\section{Experiments}
\label{sec:examples}
In this section, we demonstrate the techniques we have developed
for inference on synthetic data, as well as on a real dataset.
\subsection{Simulated data}

We consider three simulations: the ring data (Figure~\ref{fig:sim}\subref{subfig:sim_data_ring} and \subref{subfig:sim_tree_ring}),
the Mickey Mouse data (Figure~\ref{fig:sim}\subref{subfig:sim_data_mouse} and \subref{subfig:sim_tree_mouse}), 
and the yingyang data (Figure~\ref{fig:sim}\subref{subfig:sim_data_yingyang} and \subref{subfig:sim_tree_yingyang}).
The smoothing bandwidth is chosen by the Silverman reference rule \cite{silverman1986density}
and we pick the significance level $\alpha= 0.05$.

{\bf Example 1: The ring data.}  (Figure~\ref{fig:sim}\subref{subfig:sim_data_ring} and \subref{subfig:sim_tree_ring})
The ring data consists of two structures: an outer ring and a center node.
The outer circle consists of $1000$ points and the central node contains $200$ points.
To construct the tree, we used $h=0.202$. 

{\bf Example 2: The Mickey Mouse data.} (Figure~\ref{fig:sim}\subref{subfig:sim_data_mouse} and \subref{subfig:sim_tree_mouse})
The Mickey Mouse data has three components: the top left and right uniform circle ($400$ points each)
and the center circle ($1200$ points).
In this case, we select $h=0.200$. 

{\bf Example 3: The yingyang data.} (Figure~\ref{fig:sim}\subref{subfig:sim_data_yingyang} and \subref{subfig:sim_tree_yingyang})
This data has $5$ connected components:
outer ring ($2000$ points), the two moon-shape regions 
($400$ points each), and the two nodes ($200$ points each).
We choose $h=0.385$. 

Figure~\ref{fig:sim} shows those data (\subref{subfig:sim_data_ring}, \subref{subfig:sim_data_mouse}, and \subref{subfig:sim_data_yingyang}) along with 
the pruned density trees (solid parts in \subref{subfig:sim_tree_ring}, \subref{subfig:sim_tree_mouse}, and \subref{subfig:sim_tree_yingyang}).
Before pruning the tree (both solid and dashed parts), there are 
more leaves than the actual number of connected components. 
But after pruning (only the solid parts), every leaf corresponds to an actual connected component.
This demonstrates the power of a good pruning procedure.

\subsection{GvHD dataset}

\begin{figure}
	\centering
	\begin{subfigure}{0.4\linewidth}
		\centering	
		\includegraphics[height=1.7 in]{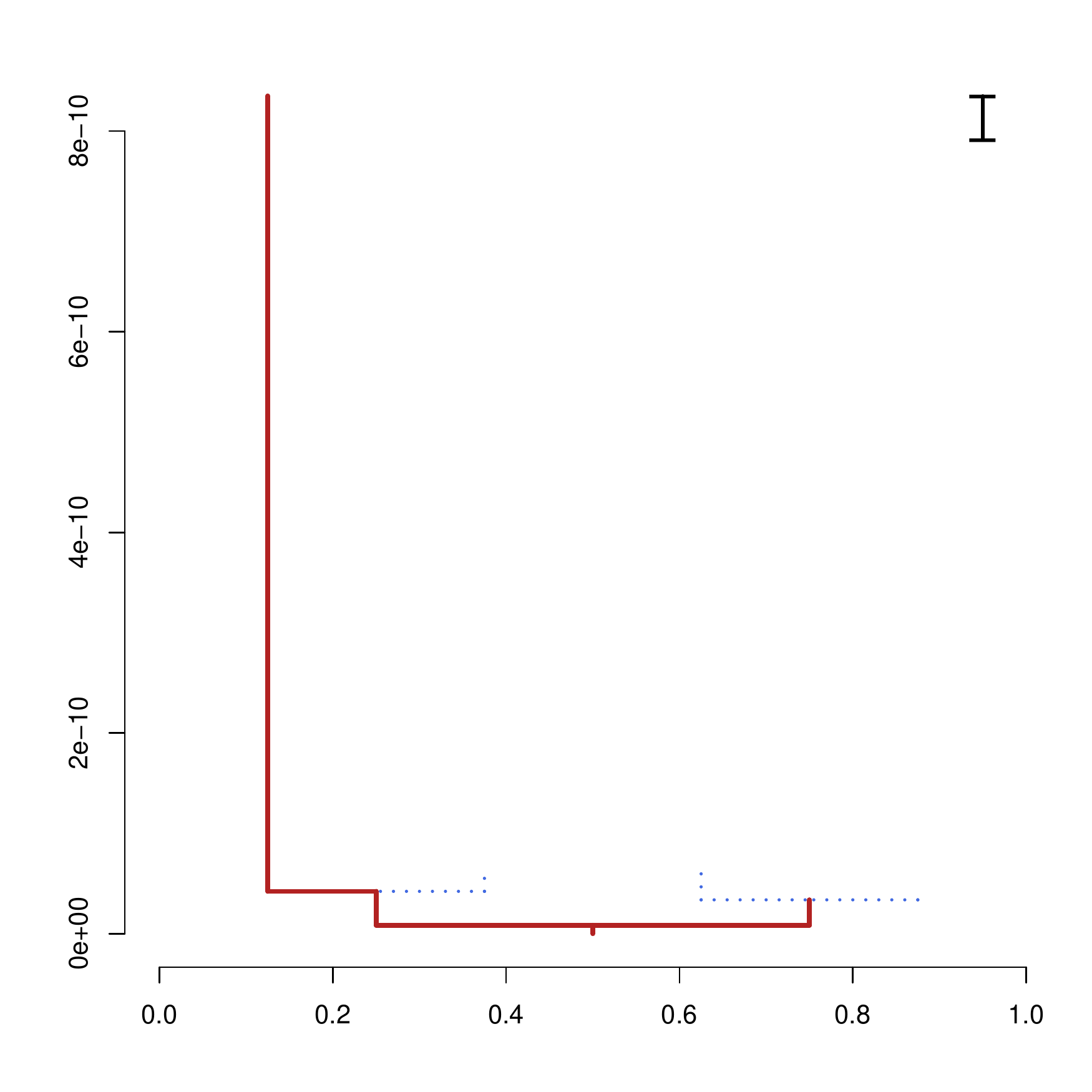}
		\caption{\label{subfig:GvHD_positive} The positive treatment data.}
	\end{subfigure}
	\begin{subfigure}{0.4\linewidth}	
		\centering
		\includegraphics[height=1.7 in]{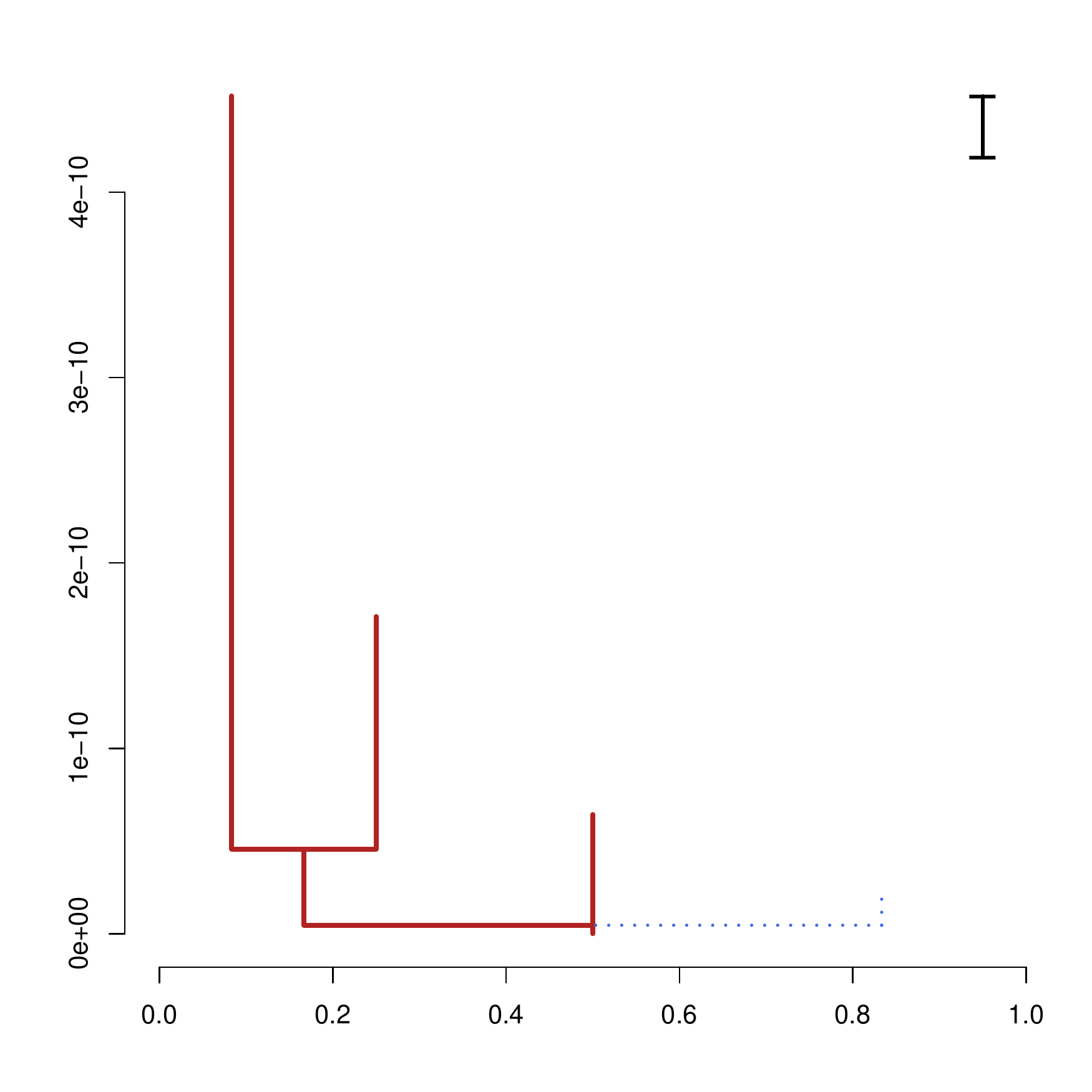}
		\caption{\label{subfig:GvHD_control} The control data.}
	\end{subfigure}
	\caption{The GvHD data.
		The solid brown lines are the remaining branches after pruning;
		the blue dashed lines are the pruned leaves (or edges).
		A bar of length $2\hat{t}_\alpha$ is at the top right corner.}
	\label{fig:GvHD}
\end{figure}

Now we apply our method to the GvHD (Graft-versus-Host Disease) dataset \cite{brinkman2007high}.
GvHD is a complication that may occur when transplanting bone marrow or stem cells
from one subject to another \cite{brinkman2007high}. 
We obtained the GvHD dataset from R package `mclust'. 
There are two subsamples: the control sample and the positive (treatment) sample.
The control sample consists of $9083$ observations and the positive sample contains $6809$ observations 
on $4$ biomarker measurements ($d=4$). 
By the normal reference rule \cite{silverman1986density}, 
we pick $h = 39.1$ for the positive sample
and $h = 42.2$ for the control sample.
We set the significance level $\alpha=0.05$.

Figure~\ref{fig:GvHD} shows the density trees in both samples.
The solid brown parts are the remaining components of density trees after pruning 
and the dashed blue parts are the branches removed by pruning. 
As can be seen, the pruned density tree of the positive sample (Figure~\ref{fig:GvHD}\subref{subfig:GvHD_positive}) is quite
different from the pruned tree of the control sample (Figure~\ref{fig:GvHD}\subref{subfig:GvHD_control}).
The density function of the positive sample has fewer bumps ($2$ significant leaves)
than the control sample ($3$ significant leaves).
By comparing the pruned trees, we can see how the two distributions differ from each other.

\section{Discussion}
\label{sec:discussion}

There are several open questions that we will address in future work.
First, it would be useful to have an algorithm that can find all trees in the confidence set that
are minimal with respect to the partial order $\preceq$.
These are the simplest trees consistent with the data.
Second, we would like to find a way to derive valid confidence sets using the metric $\dmm$
which we view as an appealing metric for tree inference.
Finally, 
we have used the Silverman reference rule \cite{silverman1986density} for choosing the bandwidth
but we would like to find a bandwidth selection method that is more targeted to tree inference.

\newpage

{\small
	\bibliographystyle{abbrvnat}
	\bibliography{nips_tree}}

\clearpage 

\appendix

\section{Topological Preliminaries}
\label{app:topology}

The goal of this section is to define an appropriate topology on the
cluster tree $\treef$ in Definition \ref{def:tree}. Defining
an appropriate topology for the cluster tree $\treef$ is important in
this paper for several reasons: (1) the topology gives geometric
insight for the cluster tree, (2) homeomorphism (topological
equivalence) is connected to equivalence in the partial order
$\preceq$ in Definition \ref{def:partial}, and (3) the topology gives
a justification for using a fixed bandwidth $h$ for constructing confidence
set $\hat{C}_{\alpha}$ as in Lemma \ref{lem:morse} to obtain faster
rates of convergence.

We construct the topology of the cluster tree $\treef$ by imposing
a topology on the corresponding collection of connected components
$\colltf$ in Definition \ref{def:tree}. For defining a topology on
$\colltf$, we define the tree distance function $d_{T_{f}}$ in
Definition \ref{def:treedistance}, and impose the metric topology induced
from the tree distance function. Using a distance function for topology
not only eases formulating topology but also enables us to inherit all
the good properties of the metric topology.

The desired tree distance function
$d_{T_{f}}:\colltf\times\colltf\to[0,\infty)$ is based on the merge
height function $m_{f}$ in Definition \ref{def:mergeheight}. For
later use in the proof, we define the tree distance function
$d_{T_{f}}$ on both $\mathcal{X}$ and $\colltf$ as follows:

\begin{definition}
	\label{def:treedistance}
	Let $f:\mathcal{X}\to[0,\infty)$ be a function, and $T_{f}$ be its cluster tree in Definition \ref{def:tree}.
	For any two points $x,y\in\mathcal{X}$, the tree distance function $d_{T_{f}}:\mathcal{X}\times\mathcal{X}\to[0,\infty)$ of $T_{f}$ on $\mathcal{X}$ is defined as 
	\begin{align*}
	d_{T_{f}}(x,y)=f(x)+f(y)-2m_{f}(x,y). 
	\end{align*}
	Similarly, for any two clusters $C_{1}, C_2 \in \colltf$, we first define $\lambda_1 = \sup
	\{\level: C_1 \in \treef(\level)\}$, and $\lambda_2$ analogously. We then define the tree distance function $d_{T_{f}}:\colltf\times\colltf\to[0,\infty)$ of $T_{f}$ on $\mathcal{X}$ as: 
	\begin{align*}
	d_{T_{f}}(C_{1},C_{2})=\lambda_{1}+\lambda_{2}-2m_{f}(C_{1},C_{2}).
	\end{align*}
\end{definition}
The tree distance function $d_{T_{f}}$ in Definition \ref{def:mergeheight} is a pseudometric on $\mathcal{X}$ and is a metric on $\colltf$ as desired, proven in Lemma \ref{lem:treedistance_metric}. The proof is given later in Appendix \ref{app:proof_topology_partial}.
\begin{lemma}
	\label{lem:treedistance_metric}
	Let $f:\mathcal{X}\to[0,\infty)$ be a function, $T_{f}$ be its cluster tree in Definition \ref{def:tree}, and $d_{T_{f}}$ be its tree distance function in Definition \ref{def:treedistance}. Then $d_{T_{f}}$ on $\mathcal{X}$ is a pseudometric and $d_{T_{f}}$ on $\colltf$ is a metric.
\end{lemma}
From the metric $d_{T_{f}}$ on $\colltf$ in Definition \ref{def:treedistance}, we impose the induced metric topology on $\colltf$. We say $T_{f}$ is homeomorphic to $T_{g}$, or
$T_{f}\cong T_{g}$, when their corresponding collection of connected components are homeomorphic,
i.e. $\colltf\cong \colltg$.  (Two spaces are homeomorphic if there
exists a bijective continuous function between them, with a continuous inverse.)

To get some geometric understanding of the cluster tree in Definition
\ref{def:tree}, we identify edges that constitute the cluster
tree. Intuitively, edges correspond to either leaves or internal
branches. An edge is roughly defined as a set of clusters whose
inclusion relationship with respect to clusters outside an edge are
equivalent, so that when the collection of connected components is
divided into edges, we observe the same inclusion relationship between
representative clusters whenever any cluster is selected as
a representative for each edge.

For formally defining edges, we define an interval in the cluster
tree and the equivalence relation in the cluster tree. For any two
clusters $A,B\in\colltf$, the interval $[A,B]\subset\colltf$ is
defined as a set clusters that contain $A$ and are contained in $B$,
i.e.
\[
[A,B]:=\left\{ C\in\colltf:\,A\subset C\subset B\right\} ,
\]
The equivalence relation $\sim$ is defined as $A\sim B$ if and only if their inclusion relationship with respect to clusters outside $[A,B]$ and $[B,A]$, i.e.
\begin{align*}
& A\sim B\text{ if and only if}\\
& \text{for all }C\in\colltf\text{ such that }C\notin[A,B]\cup[B,A],\text{ }C\subset A\text{ iff }C\subset B\text{ and }A\subset C\text{ iff }B\subset C.
\end{align*}
Then it is easy to see that the relation $\sim$ is reflexive($A\sim
A$), symmetric$(A\sim B$ implies $B\sim A$), and transitive ($A\sim B$
and $B\sim C$ implies $A\sim C$). Hence the relation $\sim$ is indeed
an equivalence relation, and we can consider the set of equivalence
classes $\colltf/_{\sim}$. We define the edge set $E(T_{f})$ as
$E(T_{f}):=\colltf/_{\sim}$.

For later use, we define the partial order on the edge set $E(T_{f}))$ as follows: 
$[C_{1}]\leq[C_{2}]$ if and only if for all $A\in[C_{1}]$ and $B\in[C_{2}]$,
$A\subset B$. We say that a tree $T_{f}$ is finite if its edge $E(T_{f})$ is a finite set.


\section{The Partial Order}
\label{app:partial}

As discussed in Section \ref{sec:background}, to see that the partial
order $\preceq$ in Definition \ref{def:partial} is indeed a partial
order, we need to check the reflexivity, the transitivity, and the
antisymmetry. The reflexivity and the transitivity are easier to
check, but to show antisymmetric, we need to show that if two trees
$T_{f}$ and $T_{g}$ satisfies $T_{f}\preceq T_{g}$ and $T_{g}\preceq
T_{f}$, then $T_{f}$ and $T_{g}$ are equivalent in some sense. And we
give the equivalence relation as the topology on the cluster tree
defined in Appendix \ref{app:topology}. The argument is formally
stated in Lemma \ref{lem:partial_equivalence}. The proof is done
later in Appendix \ref{app:proof_topology_partial}.

\begin{lemma}
	\label{lem:partial_equivalence}
	Let $f,g:\mathcal{X}\rightarrow[0,\infty)$ be functions, and $T_{f},T_{g}$ be their cluster trees  in Definition \ref{def:tree}. Then if $f,g$ are continuous and $T_{f},T_{g}$ are finite, $T_{f}\preceq T_{g}$ and $T_{g}\preceq T_{f}$
	implies that there exists a homeomorphism $\Phi:\colltf\rightarrow\colltg$
	that preserves the root, i.e. $\Phi(\mathcal{X})=\mathcal{X}$. Conversely,
	if there exists a homeomorphism $\Phi:\colltf\rightarrow\colltg$
	that preserves the root, $T_{f}\preceq T_{g}$ and $T_{g}\preceq T_{f}$
	hold.
\end{lemma}

The partial order $\preceq$ in Definition \ref{def:partial} gives
a formal definition of simplicity of trees, and it is used to justify pruning
schemes in Section \ref{subsec:pruning}. Hence it is important to
match the partial order $\preceq$ with the intuitive notions of the
complexity of the tree. We provided three arguments in Section
\ref{sec:background}: (1) if $T_{f}\preceq T_{g}$ holds then it must
be the case that $\text{(number of edges of }T_{f})\leq\text{(number
	of edges of }T_{g})$, (2) if $T_{g}$ can be obtained from $T_{f}$ by
adding edges, then $T_{f}\preceq T_{g}$ holds, and (3) the existence
of a topology preserving embedding from $\colltf$ to $\colltg$ implies
the relationship $\treef \preceq \treeg$. We formally state each item
in Lemma \ref{lem:partial_edgenumber}, \ref{lem:partial_insert}, and
\ref{lem:partial_embedding}. Proofs of these lemmas are done later in
Appendix \ref{app:proof_topology_partial}.

\begin{lemma}
	\label{lem:partial_edgenumber}
	Let $f,g:\mathcal{X}\rightarrow[0,\infty)$ be functions, and $T_{f},T_{g}$ be their cluster trees  in Definition \ref{def:tree}. Suppose $T_{f}\preceq T_{g}$
	via $\Phi:\colltf\to\colltg$. Define $\bar{\Phi}:E(T_{f})\to E(T_{g})$
	by for $[C]\in E(T_{f})$ choosing any $C\in[C]$
	and defining as $\bar{\Phi}([C])=[\Phi(C)]$. Then $\bar{\Phi}$
	is injective, and as a consequence, $|E(T_{f})|\leq|E(T_{g})|$.	
\end{lemma}

\begin{lemma}
	\label{lem:partial_insert}
	Let $f,g:\mathcal{X}\rightarrow[0,\infty)$ be functions, and $T_{f},T_{g}$ be their cluster trees  in Definition \ref{def:tree}. If $T_{g}$ can
	be obtained from $T_{f}$ by adding edges, then $T_{f}\preceq T_{g}$
	holds.
\end{lemma}

\begin{lemma}
	\label{lem:partial_embedding}
	Let $f,g:\mathcal{X}\rightarrow[0,\infty)$ be functions, and $T_{f},T_{g}$ be their cluster trees  in Definition \ref{def:tree}. If there exists
	a one-to-one map $\Phi:\colltf\to\colltg$ that is a homeomorphism
	between $\colltf$ and $\Phi(\colltf)$ and preserves the root, i.e. $\Phi(\mathcal{X})=\mathcal{X}$,
	then $T_{f}\preceq T_{g}$ holds.
\end{lemma}

\section{Hadamard Differentiability}
\label{app:hadamard}
\begin{definition} 
	[see page 281 of \cite{wellner2013weak}]
	Let $\mathbb{D}$ and $\mathbb{E}$ be normed spaces and let
	$\phi: \mathbb{D}_\phi \to \mathbb{E}$ be a map
	defined on a subset $\mathbb{D}_\phi\subset \mathbb{D}$.
	Then $\phi$ is Hadamard differentiable at $\theta$ if there exists a continuous, linear map
	$\phi'_\theta: \mathbb{D}\to\mathbb{E}$ such that
	$$
	\left\| \frac{\phi(\theta + t q_t)-\phi(\theta)}{t} - \phi_\theta'(h)\right\|_{\mathbb{E}}\to 0
	$$
	as $t\to 0$, for every $q_t \to q$.
\end{definition}

Hadamard differentiability is a key property for bootstrap inference since
it is a sufficient condition for the delta method; for more details,
see section 3.1 of \cite{wellner2013weak}. 
Recall that $\dmm$ is based on the
function $d_{T_{p}}(x,y) = p(x)+p(y) -2 m_p(x,y).$
The following theorem shows that the function $d_{T_{p}}$ is not Hadamard differentiable for some pairs
$(x, y)$.
In our case $\mathbb{D}$ is the set of continuous functions on the sample space,
$\mathbb{E}$ is the real line,
$\theta = p$,
$\phi(p)$ is 
$d_{T_{p}}(x,y)$
and the norm on $\mathbb{E}$ is the usual Euclidean norm.

\begin{theorem}
	\label{thm:hada}
	Let $B(x)$ be the smallest set $B\in T_p$ such that $x\in B$.
	$d_{T_{p}}(x,y)$ is not Hadamard differentiable for $x\neq y$ when
	one of the following two scenarios occurs:
	\begin{itemize}
		\item[(i)] $\min\{p(x),p(y)\}=p(c)$ for some critical point $c$.
		\item[(ii)] $B(x)=B(y)$ and $p(x)=p(y)$.
	\end{itemize}
\end{theorem}

The merge distortion metric $\dm$ is also not Hadamard differentiable.

\section{Confidence Sets Constructions}
\label{app:confidence_set}

\subsection{Regularity conditions on the kernel}
\label{app:regularity}
To apply the results in \cite{cck} which imply
that the bootstrap confidence set is consistent, we consider the following two assumptions.

\begin{description}
	\item[\textbf{(K1)}] The kernel function $K$ has the bounded second derivative and is symmetric, non-negative, and 
	$$\int x^2K(x)dx<\infty,\qquad \int K(x)^2dx<\infty.
	$$ 
	\item[\textbf{(K2)}] The kernel function $K$ satisfies
	\begin{equation}
	\begin{aligned}
	\mathcal{K} = \left\{y\mapsto K\left(\frac{x-y}{h}\right): x\in\mathbb{R}^d, h>0\right\}.
	\end{aligned}
	\end{equation}
	We require that $\mathcal{K}$ satisfies
	\begin{align}
	\underset{P}{\sup} N\left(\mathcal{K}, L_2(P), \epsilon\|F\|_{L_2(P)}\right)\leq \left(\frac{A}{\epsilon}\right)^v
	\label{eq:VC}
	\end{align}
	for some positive numbers $A$ and $v$,
	where $N(T,d,\epsilon)$ denotes the $\epsilon$-covering number of the metric space $(T,d)$, 
	$F$ is the envelope function of $\mathcal{K}$, 
	and the supremum is taken over the whole $\mathbb{R}^d$.
	The $A$ and $v$ are usually called the VC characteristics of $\mathcal{K}$.
	The norm $\|F\|^2_{L_2(P)} =\int|F(x)|^2dP(x)$.
\end{description}

Assumption (K1) is to ensure that the variance of
the KDE is bounded and $p_h$ has the bounded second derivative.
This assumption is very common in statistical literature,
see e.g. \cite{wasserman2006all,scott2015multivariate}. 
Assumption (K2) is to regularize the complexity of the kernel function
so that the supremum norm for kernel functions and their derivatives 
can be bounded in probability.
A similar assumption appears in \cite{Einmahl2005} and \cite{genovese2014nonparametric}.
The Gaussian kernel and most compactly supported kernels satisfy both assumptions.

\subsection{Pruning}
\label{app:pruning}

The goal of this section is to formally define the pruning scheme in
Section \ref{subsec:pruning}. Note that when pruning leaves and
internal branches, when the cumulative length is computed for each
leaf and internal branch, then the pruning process can be done at
once. We provide two pruning schemes in Section
\ref{subsec:pruning} in a unifying framework by defining an appropriate
notion of lifetime for each edge, and deleting all insignificant edges with
small lifetimes. To follow the pruning schemes in Section
\ref{subsec:pruning}, we require that the lifetime of a child edge is
shorter than the lifetime of a parent edge, so that we can delete edges from
the top. We evaluate the lifetime of each edge by an appropriate
nonnegative (possibly infinite) function $\life$. We formally define
the pruned tree $Pruned_{\life,\hat{t}_{\alpha}}(\hat{T}_{h})$ as
follows:

\begin{definition}
	Suppose the function $\life :E(\hat{T}_{h})\to [0,+\infty]$ satisfies that $[C_{1}]\leq[C_{2}]\Longrightarrow\life([C_{1}])\subset\life([C_{2}])$.
	We define the pruned tree $Pruned_{\life,\hat{t}_{\alpha}}(\hat{T}_{h}):\,\mathbb{R}\to2^{\mathcal{X}}$
	as 
	\[
	Pruned_{\life,\hat{t}_{\alpha}}(\hat{T}_{h})(\lambda)=\left\{ C\in\hat{T}_{h}(\lambda-\hat{t}_{\alpha}):\,\life([C])>\hat{t}_{\alpha}\right\} .
	\]
\end{definition}

We suggest two $\life$ functions corresponding to two pruning schemes in Section \ref{subsec:pruning}. We first need several definitions. For any
$[C]\in E(\hat{T}_{h})$, define its level as 
\[
level([C]):=\left\{ \lambda:\,\text{there exists }A\in[C]\cap\hat{T}_{h}(\lambda)\right\} ,
\]
and define its cumulative level as 
\[
cumlevel([C]):=\left\{ \lambda:\,\text{there exists }A\in\hat{T}_{h}(\lambda),\,B\in[C]\text{ such that }A\subset B\right\} .
\]
Then $\life^{leaf}$
corresponds to first pruning scheme in Section \ref{subsec:pruning}, which is to prune out only insignificant leaves. 
\[
\life^{leaf}([C])=\begin{cases}
\sup\{level([C])\}-\inf\{level([C])\} & \text{ if }\inf\{level([C])\}\neq\inf\left\{ cumlevel([C])\right\} \\
+\infty & \text{ otherwise}.
\end{cases}.
\]
And $\life^{top}$ corresponds to second pruning scheme in Section \ref{subsec:pruning}, which is to prune out insignificant edges from the top.
\[
\life^{top}([C])=\sup\{cumlevel([C])\}-\inf\left\{ cumlevel([C])\right\} .
\]

Note that $\life^{leaf}$ is lower bounded by $\life^{top}$. In fact, for any $\life$ function that is lower bounded by $\life^{top}$, the pruned tree $Pruned_{\life,\hat{t}_{\alpha}}$ is a valid tree in the confidence set $\hat{C}_{\alpha}$ that is simpler than the original estimate $\hat{T}_{h}$, so that the pruned tree is the desired tree as discussed in Section \ref{subsec:pruning}. We formally state as follows. The proof is given in Appendix \ref{app:proof_confidence_set}
\begin{lemma}
	\label{lem:pruning}
	Suppose that the $\life$ function satisfies: for all $[C]\in E(\hat{T}_{h})$,
	$\life^{top}([C])\leq\life([C])$. Then
	\begin{itemize}
		\item[(i)] $Pruned_{\life,\hat{t}_{\alpha}}(\hat{T}_{h}) \preceq \esttree$.
		\item[(ii)]	there exists a function $\tilde{p}$ such that $T_{\tilde{p}}=Pruned_{\life,\hat{t}_{\alpha}}(\hat{T}_{h})$.
		\item[(iii)] $\tilde{p}$ in (ii) satisfies $\tilde{p} \in \hat C_\alpha$.
	\end{itemize}	
\end{lemma}

{\bf Remark:}
It can be shown that complete pruning --- simultaneously removing all leaves and branches with length less than $2\hat t_\alpha$ ---
can in general yield a tree that is outside the confidence set.
For example, see Figure~\ref{fig:prune}.
If we do complete pruning to this tree, we will get the trivial tree.

\section{Proofs for Appendix \ref{app:topology} and \ref{app:partial}}
\label{app:proof_topology_partial}
\subsection{Proof of Lemma \ref{lem:treedistance_metric}}

\textbf{Lemma \ref{lem:treedistance_metric}.}\textit{
	Let $f:\mathcal{X}\to[0,\infty)$ be a function, $T_{f}$ be its cluster tree in Definition \ref{def:tree}, and $d_{T_{f}}$ be its tree distance function in Definition \ref{def:treedistance}. Then $d_{T_{f}}$ on $\mathcal{X}$ is a pseudometric and $d_{T_{f}}$ on $\colltf$ is a metric.}

\begin{proof} 
	
	First, we show that $d_{T_{f}}$ on $\mathcal{X}$ is a pseudometric. To do this, we need to show non-negativity($d_{T_{f}}(x,y)\geq 0$), $x=y$ implying $d_{T_{f}}(x,y) = 0$, symmetry($d_{T_{f}}(x,y) = d_{T_{f}}(y,x)$), and subadditivity($d_{T_{f}}(x,y)+d_{T_{f}}(y,z) \leq d_{T_{f}}(x,z)$).
	
	For non-negativity, note that for all $x,y\in\mathcal{X}$, $m_{f}(x,y)\leq\min\left\{ f(x),f(y)\right\} $, so 
	\begin{equation}
	\label{eq:proof_pseudometric_nonnegativity}
	d_{T_{f}}(x,y)=f(x)+f(y)-2m_{f}(x,y)\geq0.
	\end{equation}
	For $x=y$ implying $d_{T_{f}}(x,y) = 0$, $x=y$ implies $m_{f}(x,y)=f(x)=f(y)$, so 
	\begin{equation}
	\label{eq:proof_pseudometric_identity}
	x=y\Longrightarrow d_{T_{f}}(x,y)=0.
	\end{equation}
	For symmetry, since $m_{f}(x,y)=m_{f}(y,x)$, 
	\begin{equation}
	\label{eq:proof_pseudometric_symmetry}
	d_{T_{f}}(x,y)=d_{T_{f}}(y,x).
	\end{equation}
	For subadditivity, note first that $m_{f}(x,y)\leq f(y)$ and $m_{f}(y,z)\leq f(y)$ holds, so 
	\begin{equation}
	\label{eq:proof_pseudometric_maxbound}
	\max\left\{ m_{f}(x,y),\,m_{f}(y,z)\right\} \leq f(y).
	\end{equation}
	And also note that there exists $C_{xy},C_{yz}\in T_{f}\left(\min\left\{ m_{f}(x,y),\,m_{f}(y,z)\right\} \right)$
	that satisfies $x,y\in C_{xy}$ and $y,z\subset C_{yz}$. Then $y\in C_{xy}\cap C_{yz}\neq\emptyset$,
	so $x,z\in C_{xy}=C_{yz}$. Then from definition of $m_{f}(x,z)$, this implies that 
	\begin{equation}
	\label{eq:proof_pseudometric_minbound}
	\min\left\{ m_{f}(x,y),\,m_{f}(y,z)\right\} \leq m_{f}(x,z).
	\end{equation}
	And by applying \eqref{eq:proof_pseudometric_maxbound} and \eqref{eq:proof_pseudometric_minbound}, $d_{T_{f}}(x,y)+d_{T_{f}}(y,z)$ is upper bounded by $d_{T_{f}}(x,z)$ as
	\begin{align}
	\label{eq:proof_pseudometric_subadditivity}
	& d_{T_{f}}(x,y)+d_{T_{f}}(y,z)\nonumber\\
	& =f(x)+f(y)-2m_{f}(x,y)+f(y)+f(z)-2m_{f}(y,z)\nonumber\\
	& =f(x)+f(z)-2\left(\min\left\{ m_{f}(x,y),\,m_{f}(y,z)\right\} +\max\left\{ m_{f}(x,y),\,m_{f}(y,z)\right\} -f(y)\right)\nonumber\\
	& \geq f(x)+f(z)-2m_{f}(x,z)\nonumber\\
	& =d_{T_{f}}(x,z).
	\end{align}
	Hence \eqref{eq:proof_pseudometric_nonnegativity},
	\eqref{eq:proof_pseudometric_identity},
	\eqref{eq:proof_pseudometric_symmetry}, and
	\eqref{eq:proof_pseudometric_subadditivity} implies that $d_{T_{f}}$
	on $\mathcal{X}$ is a pseudometric.
	
	Second, we show that $d_{T_{f}}$ on $\colltf$ is a metric. To do this,
	we need to show non-negativity($d_{T_{f}}(x,y)\geq 0$), identity of
	indiscernibles($x=y\iff d_{T_{f}}(x,y) = 0$), symmetry($d_{T_{f}}(x,y)
	= d_{T_{f}}(y,x)$), and subadditivity($d_{T_{f}}(x,y)+d_{T_{f}}(y,z)
	\leq d_{T_{f}}(x,z)$).
	
	For nonnegativity, note that if $C_{1}\in T_{f}(\lambda_{1})$ and
	$C_{2}\in T_{f}(\lambda_{2})$, then
	$m_{f}(C_{1},C_{2})\leq\min\{\lambda_{1},\lambda_{2}\}$, so
	\begin{equation}
	\label{eq:proof_metric_nonnegativity}
	d_{T_{f}}(C_{1},C_{2})=\lambda_{1}+\lambda_{2}-2m_{f}(C_{1},C_{2})\geq0.
	\end{equation}
	For identity of indiscernibles, $C_{1}=C_{2}$ implies
	$m_{f}(C_{1},C_{2})=\lambda_{1}=\lambda_{2}$, so
	\begin{equation}
	\label{eq:proof_metric_identity1}
	C_{1}=C_{2}\Longrightarrow d_{T_{f}}(C_{1},C_{2})=0.
	\end{equation}
	And conversely, $d_{T_{f}}(C_{1},C_{2})=0$ implies
	$\lambda_{1}=\lambda_{2}=m_{f}(C_{1},C_{2})$, so there exists $C\in
	T_{f}(\lambda_{1})$ such that $C_{1}\subset C$ and $C_{2}\subset
	C$. Then since $C_{1},C_{2},C\in T_{f}(\lambda_{1})$, so $C_{1}\cap
	C\neq\emptyset$ implies $C_{1}=C$ and similarly $C_{2}=C$, so
	\begin{equation}
	\label{eq:proof_metric_identity2}
	d_{T_{f}}(C_{1},C_{2})=0\Longrightarrow C_{1}=C_{2}.
	\end{equation}
	Hence \eqref{eq:proof_metric_identity1} and \eqref{eq:proof_metric_identity2} implies identity of indiscernibles as
	\begin{equation}
	\label{eq:proof_metric_identity}
	C_{1}=C_{2}\iff d_{T_{f}}(C_{1},C_{2})=0.
	\end{equation}
	For symmetry, since $m_{f}(C_{1},C_{2})=m_{f}(C_{2},C_{1})$, 
	\begin{equation}
	\label{eq:proof_metric_symmetry}
	d_{T_{f}}(C_{1},C_{2})=d_{T_{f}}(C_{2},C_{1}).
	\end{equation}
	For subadditivity, note that $m_{f}(C_{1},C_{2})\leq\lambda_{2}$ and $m_{f}(C_{2},C_{3})\leq\lambda_{2}$
	holds, so 
	\begin{equation}
	\label{eq:proof_metric_maxbound}
	\max\left\{ m_{f}(C_{1},C_{2}),\,m_{f}(C_{2},C_{3})\right\} \leq\lambda_{2}.
	\end{equation}
	And also note that there 
	exists $C_{12},C_{23}\in T_{f}\left(\min\left\{ m_{f}(C_{1},C_{2}),\,m_{f}(C_{2},C_{3})\right\} \right)$
	that satisfies $C_{1},C_{2}\subset C_{12}$ and $C_{2},C_{3}\subset C_{23}$.
	Then $C_{2}\subset C_{12}\cap C_{23}\neq\emptyset$, so $C_{1},C_{3}\in C_{12}=C_{23}$.
	Then from definition of $m_{f}(C_{1},C_{3})$, this implies that 
	\begin{equation}
	\label{eq:proof_metric_minbound}
	\min\left\{ m_{f}(C_{1},C_{2}),\,m_{f}(C_{2},C_{3})\right\} \leq m_{f}(C_{1},C_{3}).
	\end{equation}
	And by applying \eqref{eq:proof_metric_maxbound} and 
	\eqref{eq:proof_metric_minbound}, $d_{T_{f}}(C_{1},C_{2})+d_{T_{f}}(C_{2},C_{3})$ 
	is upper bounded by $d_{T_{f}}(C_{1},C_{3})$ as
	\begin{align}
	\label{eq:proof_metric_subadditivity}
	& d_{T_{f}}(C_{1},C_{2})+d_{T_{f}}(C_{2},C_{3}) \nonumber\\
	& =\lambda_{1}+\lambda_{2}-2m_{f}(C_{1},C_{2})+\lambda_{2}+\lambda_{3}-2m_{f}(C_{2},C_{3}) \nonumber\\
	& =\lambda_{1}+\lambda_{3}-2\left(\min\left\{ m_{f}(C_{1},C_{2}),\,m_{f}(C_{2},C_{3})\right\} +
	\max\left\{ m_{f}(C_{1},C_{2}),\,m_{f}(C_{2},C_{3})\right\} -\lambda_{2}\right)\nonumber\\
	& \geq\lambda_{1}+\lambda_{3}-2m_{f}(C_{1},C_{3})\nonumber\\
	& =d_{T_{f}}(C_{1},C_{3}).
	\end{align}
	Hence \eqref{eq:proof_metric_nonnegativity},
	\eqref{eq:proof_metric_identity},
	\eqref{eq:proof_metric_symmetry}, and
	\eqref{eq:proof_metric_subadditivity} $d_{T_{f}}$ on $\colltf$ is a
	metric.
	
\end{proof}

\subsection{Proof of Lemma \ref{lem:partial_equivalence}}
\textbf{Lemma \ref{lem:partial_equivalence}.}\textit{
	Let $f,g:\mathcal{X}\rightarrow[0,\infty)$ be functions, and $T_{f},T_{g}$ be their cluster trees  in Definition \ref{def:tree}. Then if $f,g$ are continuous and $T_{f},T_{g}$ are finite, $T_{f}\preceq T_{g}$ and $T_{g}\preceq T_{f}$
	implies that there exists a homeomorphism $\Phi:\colltf\rightarrow\colltg$
	that preserves the root, i.e. $\Phi(\mathcal{X})=\mathcal{X}$. Conversely,
	if there exists a homeomorphism $\Phi:\colltf\rightarrow\colltg$
	that preserves the root, $T_{f}\preceq T_{g}$ and $T_{g}\preceq T_{f}$
	hold.
}

\begin{proof}
	
	First, we show that $T_{f}\preceq T_{g}$ and $T_{g}\preceq T_{f}$
	implies homeomorphism. Let $\Phi:\colltf\to\colltg$ be the map that
	gives the partial order $T_{f}\preceq T_{g}$ in Definition
	\ref{def:partial}.  Then from Lemma \ref{lem:partial_edgenumber},
	$\bar{\Phi}:E(T_{f})\to E(T_{g})$ is injective and
	$|E(T_{f})|\leq|E(T_{g})|$. With a similar argument,
	$|E(T_{g})|\leq|E(T_{f})|$ holds, so
	\[
	|E(T_{f})|=|E(T_{g})|.
	\]
	Since we assumed that $T_{f}$ and $T_{g}$ are finite,
	i.e. $|E(T_{f})|$ and $|E(T_{g})|$ are finite, $\bar{\Phi}$ becomes a
	bijection.
	
	Now, let $[C_{1}]$ and $[C_{2}]$ be adjacent edges in $E(T_{f})$, and
	without loss of generality, assume $C_{1}\subset C_{2}$. We argue
	below that $\bar{\Phi}([C_{1}])$ and $\bar{\Phi}([C_{2}])$ are also
	adjacent edges. Then $\Phi(C_{1})\subset\Phi(C_{2})$ holds from
	Definition \ref{def:partial}, and since $\bar{\Phi}$ is bijective,
	$[\Phi(C_{1})]=\bar{\Phi}([C_{1}])$ and
	$[\Phi(C_{2})]=\bar{\Phi}([C_{2}])$ holds. Suppose there exists
	$\tilde{C}_{3}\in\colltg$ such that
	$[\tilde{C}_{3}]\notin\{\bar{\Phi}([C_{1}]),\,\bar{\Phi}([C_{2}])\}$
	and $\Phi(C_{1})\subset\tilde{C}_{3}\subset\Phi(C_{2})$.  Then since
	$\bar{\Phi}$ is bijective, there exists $C_{3}\in\colltf$ such that
	$[\Phi(C_{3})]=[\tilde{C}_{3}]$. Then
	$\Phi(C_{1})\subset\tilde{C}_{3}\subset\Phi(C_{2})$ implies that
	$C_{1}\subset C_{3}\subset C_{2}$, and $\bar{\Phi}$ being a bijection
	implies that $[C_{3}]\notin\{[C_{1}],\,[C_{3}]\}$.  This is a
	contradiction since $[C_{1}]$ and $[C_{2}]$ are adjacent edges. Hence
	there is no such $\tilde{C}_{3}$, and $\bar{\Phi}([C_{1}])$ and
	$\bar{\Phi}([C_{2}])$ are adjacent edges. Therefore,
	$\bar{\Phi}:E(T_{f})\to E(T_{g})$ is a bijective map that sends
	adjacent edges to adjacent edges, and also sends root edge to root
	edge.
	
	Then combining $\bar{\Phi}:E(T_{f})\to E(T_{g})$ being bijective sending adjacent edges to adjacent edges and root edge to root edge, and $f,g$ being continuous, the map $\bar{\Phi}:E(T_{f})\to E(T_{g})$ can be extended to a homeomorphism $\colltg\to\colltf$
	that preserves the root.
	
	Second, the part that homeomorphism implies $T_{f}\preceq T_{g}$ and
	$T_{g}\preceq T_{f}$ follows by Lemma \ref{lem:partial_embedding}.
\end{proof}

\subsection{Proof of Lemma \ref{lem:partial_edgenumber}}

\textbf{Lemma \ref{lem:partial_edgenumber}.}\textit{
	Let $f,g:\mathcal{X}\rightarrow[0,\infty)$ be functions, and $T_{f},T_{g}$ be their cluster trees  in Definition \ref{def:tree}. Suppose $T_{f}\preceq T_{g}$
	via $\Phi:\colltf\to\colltg$. Define $\bar{\Phi}:E(T_{f})\to E(T_{g})$
	by for $[C]\in E(T_{f})$ choosing any $C\in[C]$
	and defining as $\bar{\Phi}([C])=[\Phi(C)]$. Then $\bar{\Phi}$
	is injective, and as a consequence, $|E(T_{f})|\leq|E(T_{g})|$.	
}

\begin{proof}
	
	We will first show that equivalence relation on $\colltg$ implies
	equivalence relation on $\colltf$, i.e.
	\begin{equation}
	\label{eq:proof_injective_equivalence}
	\Phi(C_{1})\sim\Phi(C_{2})\Longrightarrow C_{1}\sim C_{2}.
	\end{equation}
	Suppose $\Phi(C_{1})\sim\Phi(C_{2})$ in $\colltg$. Then from
	Definition \ref{def:partial} of $\Phi$, for any $C\in\colltf$ such
	that $C\notin[C_{1},C_{2}]\cup[C_{2},C_{1}]$,
	$\Phi(C)\notin[\Phi(C_{1}),\Phi(C_{2})]\cup[\Phi(C_{2}),\Phi(C_{1})]$
	holds. Then from definition of $\Phi(C_{1})\sim\Phi(C_{2})$,
	\[
	\Phi(C)\subset\Phi(C_{1})\text{ iff }\Phi(C)\subset\Phi(C_{2})
	\text{ and }\Phi(C_{1})\subset\Phi(C)\text{ iff }\Phi(C_{2})\subset\Phi(C).
	\]
	Then again from Definition \ref{def:partial} of $\Phi$, equivalence
	relation holds for $C_{1}$ and $C_{2}$ holds as well, i.e.
	\[
	C\subset C_{1}\text{ iff }C\subset C_{2}\text{ and }C_{1}\subset C\text{ iff }C_{2}\subset C.
	\]
	Hence \eqref{eq:proof_injective_equivalence} is shown, and this implies that 
	\begin{align*}
	\bar{\Phi}([C_{1}])=\bar{\Phi}([C_{2}]) & \Longrightarrow[\Phi(C_{1})]=[\Phi(C_{2})]\\
	& \Longrightarrow\Phi(C_{1})\sim\Phi(C_{2})\\
	& \Longrightarrow C_{1}\sim C_{2}\\
	& \Longrightarrow[C_{1}]=[C_{2}],
	\end{align*}
	so $\bar{\Phi}$ is injective.
\end{proof}

\subsection{Proof of Lemma \ref{lem:partial_insert}}

\textbf{Lemma \ref{lem:partial_insert}.}\textit{
	Let $f,g:\mathcal{X}\rightarrow[0,\infty)$ be functions, and $T_{f},T_{g}$ be their cluster trees  in Definition \ref{def:tree}. If $T_{g}$ can
	be obtained from $T_{f}$ by adding edges, then $T_{f}\preceq T_{g}$
	holds.
}

\begin{proof}
	Since $T_{g}$ can be obtained from $T_{f}$ by adding edges, there
	is a map $\Phi:T_{f}\to T_{g}$ which preserves order, i.e. $C_{1}\subset C_{2}$
	if and only if $\Phi(C_{1})\subset\Phi(C_{2})$. Hence $T_{f}\preceq T_{g}$
	holds.
\end{proof}

\subsection{Proof of Lemma \ref{lem:partial_embedding}}

\textbf{Lemma \ref{lem:partial_embedding}.}\textit{
	Let $f,g:\mathcal{X}\rightarrow[0,\infty)$ be functions, and $T_{f},T_{g}$ be their cluster trees  in Definition \ref{def:tree}. If there exists
	a one-to-one map $\Phi:\colltf\to\colltg$ that is a homeomorphism
	between $\colltf$ and $\Phi(\colltf)$ and preserves root, i.e. $\Phi(\mathcal{X})=\mathcal{X}$,
	then $T_{f}\preceq T_{g}$ holds.
}

\begin{proof}
	For any $C\in\colltf$, note that $[C,\mathcal{X}]\subset\colltf$
	is homeomorphic to an interval, hence $\Phi([C,\mathcal{X}])\subset\colltg$
	is also homeomorphic to an interval. Since $\colltg$ is topologically
	a tree, an interval in a tree with fixed boundary points is uniquely
	determined, i.e. 
	\begin{equation}
	\label{eq:proof_homeopartial_intervalimage}
	\Phi([C,\mathcal{X}])=[\Phi(C),\Phi(\mathcal{X})]=[\Phi(C),\mathcal{X}].
	\end{equation}
	For showing $T_{f}\preceq T_{g}$, we need to argue that for all $C_{1},C_{2}\in\colltf$, $C_{1}\subset C_{2}$ holds if and only if $\Phi(C_{1})\subset \Phi(C_{2})$. For only if direction, suppose $C_{1}\subset C_{2}$. Then $C_{2}\in[C_{1},\mathcal{X}]$,
	so Definition \ref{def:partial} and \eqref{eq:proof_homeopartial_intervalimage} implies
	\[
	\Phi(C_{2})\subset\Phi([C_{1},\mathcal{X}])=[\Phi(C_{1}),\mathcal{X}].
	\]
	And this implies
	\begin{equation}
	\label{eq:proof_homeopartial_onlyif}
	\Phi(C_{1})\subset\Phi(C_{2}).
	\end{equation}
	For if direction, suppose $\Phi(C_{1})\subset\Phi(C_{2})$. Then since $\Phi^{-1}:\Phi(\colltf)\to\colltf$
	is also an homeomorphism with $\Phi^{-1}(\mathcal{X})=\mathcal{X}$,
	hence by repeating above argument, we have 
	\begin{equation}
	\label{eq:proof_homeopartial_if}
	C_{1}=\Phi^{-1}(\Phi(C_{1}))\subset\Phi^{-1}(\Phi(C_{2}))=C_{2}.
	\end{equation}
	Hence \eqref{eq:proof_homeopartial_onlyif} and \eqref{eq:proof_homeopartial_if} implies $\treef \preceq \treeg$.
\end{proof}

\section{Proofs for Section \ref{sec:metric} and Appendix \ref{app:hadamard}}
\label{app:proof_metric_hadamard}

\subsection{Proof of Lemma \ref{lem:infty_M} and extreme cases}
\textbf{Lemma \ref{lem:infty_M}.} \textit{
	For any densities $p$ and $q$, the following relationships hold:
	\begin{enumerate}
		\item[(i)] When $p$ and $q$ are continuous, then $d_\infty(T_p,T_q) = \dm(T_p,T_q)$.
		\item[(ii)] $\dmm(T_{p},T_{q})\leq4d_{\infty}(T_{p},T_{q}).$
		\item[(iii)] $\dmm(T_{p},T_{q})\geq d_{\infty}(T_{p},T_{q})-a$, where $a$ is defined as above. Additionally when  
		$\mu(\mathcal{X}) = \infty$, then $\dmm(T_{p},T_{q})\geq d_{\infty}(T_{p},T_{q})$.
	\end{enumerate}
}

\begin{proof}

	(i)
	
	First, we show $d_{M}(T_{p},\,T_{q})\leq d_{\infty}(T_{p},\,T_{q})$.
	Note that this part is implicitly shown in \citet[Proof of Theorem 6]{eldridge2015beyond}.
	For all $\epsilon>0$ and for any $x,y\in\mathcal{X}$, let $C_{0}\in\treep(m_{p}(x,y)-\epsilon)$
	with $x,y\in C_{0}$. Then for all $z\in C_{0}$, $q(z)$ is lower
	bounded as 
	\begin{align*}
	q(z) & >p(z)-d_{\infty}(T_{p},T_{q})\\
	& \geq m_{p}(x,y)-\epsilon-d_{\infty}(T_{p},T_{q}),
	\end{align*}
	so $C_{0}\subset q^{-1}\left(m_{p}(x,y)-\epsilon-d_{\infty}(T_{p},T_{q}),\,\infty\right)$
	and $C_{0}$ is connected, so $x$ and $y$ are in the same connected
	component of $q^{-1}\left(m_{p}(x,y)-\epsilon-d_{\infty}(T_{p},T_{q}),\,\infty\right)$,
	which implies 
	\begin{equation}
	m_{q}(x,y)\leq m_{p}(x,y)-\epsilon-d_{\infty}(T_{p},T_{q}).
	\label{eq:proof_equivalence_dinftydM_mergebound1}
	\end{equation}
	A similar argument holds for other direction as 
	\begin{equation}
	m_{p}(x,y)\leq m_{q}(x,y)-\epsilon-d_{\infty}(T_{p},T_{q}),
	\label{eq:proof_equivalence_dinftydM_mergebound2}
	\end{equation}
	so \eqref{eq:proof_equivalence_dinftydM_mergebound1} and \eqref{eq:proof_equivalence_dinftydM_mergebound2}
	being held for all $\epsilon>0$ implies 
	\begin{equation}
	|m_{p}(x,y)-m_{q}(x,y)|\leq d_{\infty}(T_{p},T_{q}).
	\label{eq:proof_dinftydM_mergebound}
	\end{equation}
	And taking $\sup$ over all $x,y\in\mathcal{X}$ in \eqref{eq:proof_dinftydM_mergebound} $d_{M}(T_{p},\,T_{q})$ is upper bounded by $d_{\infty}(T_{p},\,T_{q})$, i.e. 
	\begin{equation}
	d_{M}(T_{p},\,T_{q})\leq d_{\infty}(T_{p},\,T_{q}).\label{eq:proof_equivalence_dinftydM_dMupperbound}
	\end{equation}
	Second, we show $d_{M}(T_{p},\,T_{q})\geq d_{\infty}(T_{p},\,T_{q})$.
	For all $\epsilon>0$, Let $x$ be such that $|p(x)-q(x)|>d_{\infty}(T_{p},T_{q})-\frac{\epsilon}{2}$.
	Then since $p$ and $q$ are continuous, there exists $\delta>0$
	such that 
	\[
	B(x,\,\delta)\subset p^{-1}\left(p(x)-\frac{\epsilon}{2},\,\infty\right)\cap q^{-1}\left(q(x)-\frac{\epsilon}{2},\,\infty\right).
	\]
	Then for any $y\in B(x,\,\delta)$, since $B(x,\,\delta)$ is connected,
	$p(x)-\frac{\epsilon}{2}\leq m_{p}(x,y)\leq p(x)$ holds and $q(x)-\frac{\epsilon}{2}\leq m_{q}(x,y)\leq q(x)$,
	so 
	\begin{align*}
	|m_{p}(x,y)-m_{q}(x,y)| & \geq|p(x)-q(x)|-\frac{\epsilon}{2}\\
	& >d_{\infty}(T_{p},T_{q})-\epsilon.
	\end{align*}
	Since this holds for any $\epsilon>0$, $d_{M}(T_{p},\,T_{q})$ is
	lower bounded by $d_{\infty}(T_{p},\,T_{q})$, i.e. 
	\begin{equation}
	d_{M}(T_{p},\,T_{q})\geq d_{\infty}(T_{p},\,T_{q}).\label{eq:proof_equivalence_dinftydM_dMlowerbound}
	\end{equation}
	\eqref{eq:proof_equivalence_dinftydM_dMupperbound} and \eqref{eq:proof_equivalence_dinftydM_dMlowerbound}
	implies $d_{\infty}(T_{p},T_{q})=\dm(T_{p},T_{q})$. 
	
	(ii)
	
	We have already seen that for all $x,y\in\mathcal{X}$, $|m_{p}(x,y)-m_{q}(x,y)|\leq d_{\infty}(T_{p},T_{q})$
	in \eqref{eq:proof_dinftydM_mergebound}. Hence for all $x,y\in\mathcal{X}$, 
	\begin{align*}
	& |[p(x)+p(y)-2m_{p}(x,y)]-[q(x)+q(y)-2m_{q}(x,y)]|\\
	& \leq|p(x)-q(x)|+|p(y)-q(y)|+2|m_{p}(x,y)-m_{q}(x,y)|\\
	& \leq4d_{\infty}(T_{p},T_{q}).
	\end{align*}
	Since this holds for all $x,y\in\mathcal{X}$, so 
	\[
	\dmm(T_{p},T_{q})\leq4d_{\infty}(T_{p},T_{q}).
	\]

	(iii)
	
	For all $\epsilon>0$, Let $x$ be such that $|p(x)-q(x)|>d_{\infty}(T_{p},T_{q})-\frac{\epsilon}{2}$,
	and without loss of generality assume that $p(x)>q(x)$. Let $y$
	be such that $p(y)+q(y)<\underset{x}{\inf}(p(x)+q(x))+\frac{\epsilon}{2}$.
	Then $m_{p}(x,y)\leq p(y)$ holds, and since $\mathcal{X}$ is connected,
	$q_{\inf}\leq m_{q}(x,y)$ holds. Hence 
	\begin{align*}
	& [p(x)+p(y)-2m_{p}(x,y)]-[q(x)+q(y)-2m_{q}(x,y)]\\
	& \geq[p(x)+p(y)-2p(y)]-[q(x)+q(y)-2q_{\inf}]\\
	& =p(x)-q(x)-(p(y)+q(y)-2q_{\inf})\\
	& >d_{\infty}(T_{p},\,T_{q})-\left(\underset{x}{\inf}(p(x)+q(x))-2q_{\inf}\right)-\epsilon\\
	& \geq d_{\infty}(T_{p},\,T_{q})-a-\epsilon,
	\end{align*}
	where $a=\underset{x\in\mathcal{X}}{\inf}(p(x)+q(x))-2\min\left\{ p_{\inf},\,q_{\inf}\right\}$.
	Since this holds for all $\epsilon>0$, we have
	\[
	\dmm(T_{p},\,T_{q})\geq d_{\infty}(T_{p},\,T_{q})-a.
	\]
\end{proof}

Hence $0\leq \dmm(T_{p},T_{q})\leq 4d_{\infty}(T_{p},T_{q})$ holds. And both extreme cases can happen, i.e. $\dmm(T_{p},T_{q})=4d_{\infty}(T_{p},T_{q})>0$ and $\dmm(T_{p},T_{q})=0,\,d_{\infty}(T_{p},T_{q})> 0$ can happens. 

\begin{lemma}
	There exists densities $p,q$ for both $\dmm(T_{p},T_{q})=4d_{\infty}(T_{p},T_{q})>0$ and $\dmm(T_{p},T_{q})=0,\,d_{\infty}(T_{p},T_{q})> 0$.
\end{lemma}

\begin{proof}
	Let $\mathcal{X}=\mathbb{R}$, $p(x)=I(x\in[0,1])$ and $q(x)=2I\left(x\in\left[0,\,\frac{1}{4}\right]\right)+2I\left(x\in\left[\frac{3}{4},\,1\right]\right)$.
	Then $d_{\infty}(T_{p},T_{q})=1$. And with $x=\frac{1}{8}$ and $y=\frac{7}{8}$,
	\begin{align*}
	|[p(x)+p(y)-2m_{p}(x,y)]-[q(x)+q(y)-2m_{q}(x,y)]| & =\left|[1+1-2]-[2+2-0]\right|\\
	& =4,
	\end{align*}
	hence $\dmm(T_{p},T_{q})=4d_{\infty}(T_{p},T_{q})$.
	
	Let $\mathcal{X}=[0,1)$, $p(x)=2I\left(x\in\left[0,\,\frac{1}{2}\right)\right)$
	and $q(x)=2I\left(x\in\left[\frac{1}{2},\,1\right)\right)$. Then
	$d_{\infty}(T_{p},T_{q})=2$. And for any $x\in\left[0,\,\frac{1}{2}\right)$
	and $y\in\left[\frac{1}{2},\,1\right)$, 
	\begin{align*}
	|[p(x)+p(y)-2m_{p}(x,y)]-[q(x)+q(y)-2m_{q}(x,y)]| & =\left|(2+0-0)+(0+2-0)\right|\\
	& =0.
	\end{align*}
	A similar case holds for $x\in\left[\frac{1}{2},\,1\right)$ and $y\in\left[0,\,\frac{1}{2}\right)$.
	And for any $x,y\in\left[0,\,\frac{1}{2}\right)$, 
	\begin{align*}
	|[p(x)+p(y)-2m_{p}(x,y)]-[q(x)+q(y)-2m_{q}(x,y)]| & =\left|(2+2-4)+(0+0-0)\right|\\
	& =0.
	\end{align*}
	and a similar case holds for $x,y\in\left[\frac{1}{2},\,1\right)$.
	Hence $\dmm(T_{p},T_{q})=0$.
\end{proof}

\subsection{Proof of Theorem~\ref{thm:hada}}

\textbf{Theorem \ref{thm:hada}.} \textit{
	Let $B(x)$ be the smallest set $B\in T_p$ such that $x\in B$.
	$d_{T_{p}}(x,y)$ is not Hadamard differentiable for $x\neq y$ when
	one of the following two scenarios occurs:
	\begin{itemize}
		\item[(i)] $\min\{p(x),p(y)\}=p(c)$ for some critical point $c$.
		\item[(ii)] $B(x)=B(y)$ and $p(x)=p(y)$.
	\end{itemize}
}

\begin{proof}
	For $x,y\in\mathbb{K}$, note that the merge height satisfies
	$$
	m_p(x,y) = \min\{t:(x,y) \,\mbox{are in the same connected component of} \, L(t) \} .
	$$
	Recall that
	$$
	d_{T_{p}}(x,y) = p(x)+p(y) -2m_p(x,y).
	$$
	Note that the modified merge distortion metric is $\dmm(p,q) = \sup_{x,y}|d_{T_{p}}(x,y)-d_{T_{q}}(x,y)|$.

	A feature of the merge height is that
	\begin{align*}
	m_p(x,y) = p(x) &\Rightarrow B(y)\subset B(x) \\
	m_p(x,y) = p(y) &\Rightarrow B(x)\subset B(y)\\
	m_p(x,y) \neq p(y) \mbox{ or }p(x) &\Rightarrow \exists c(x,y)\in\mathcal{C} \,\,\, s.t.\,\,\,m_p(x,y)=p(c(x,y)).
	\end{align*}
	where $\mathcal{C}$ is the collection of all critical points.
	Thus,
	we have
	$$
	d_{T_{p}}(x,y) =
	\begin{cases}
	p(x)-p(y)       & \quad \text{if } B(y)\subset B(x)\\
	p(y)-p(x)       & \quad \text{if } B(x)\subset B(y)\\
	p(x)+p(y) -2p(c(x,y))&\quad \text{otherwise}
	\end{cases}.
	$$

	\begin{figure}
		\centering
		\includegraphics[height=2 in]{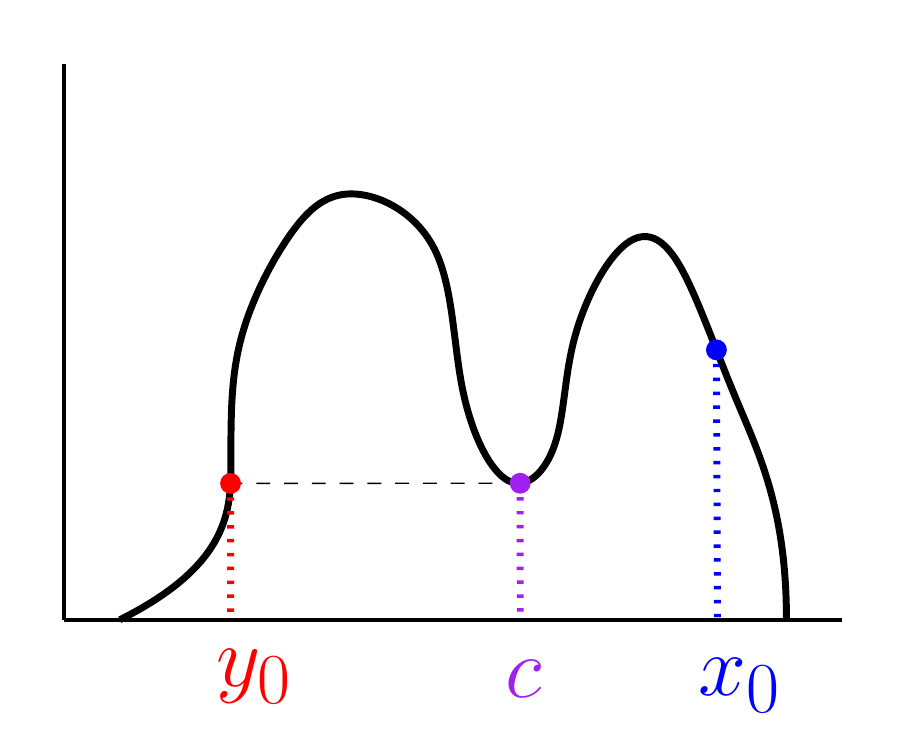}
		\caption{The example used in the proof of Theorem~\ref{thm:hada}.}
		\label{fig:thm:hahd1}
	\end{figure}
	
	{\bf Case 1:}\\
	We pick a pair of $x_0,y_0$ as in Figure~\ref{fig:thm:hahd1}.
	Now we consider a smooth symmetric function $g(x)>0$
	such that it peaks at $0$ and monotonically decay and has support $[-\delta,\delta]$
	for some small $\delta>0$.
	We pick $\delta$ small enough such that
	$p_\epsilon(x_0)=p(x_0), p_\epsilon(y_0)=p(y_0)$.
	For simplicity, let $g(0) = \max_x g(x) = 1$.
	
	Now consider perturbing $p(x)$ along $g(x-c)$ with amount $\epsilon$.
	Namely, we define 
	$$
	p_\epsilon(x)  = p(x) + \epsilon\cdot g(x-c).
	$$
	For notational convenience,
	define
	$\xi_{p,\epsilon} = d_{T_{p_\epsilon}}(x_0,y_0)$.
	When $|\epsilon|$ is sufficiently small,
	define
	\begin{align*}
	\xi_{p,\epsilon}(x_0,y_0) &= d_{T_{p}}(x_0,y_0) \quad \mbox{if $\epsilon>0$},\\
	\xi_{p,\epsilon}(x_0,y_0) &= d_{T_{p}}(x_0,y_0)-2 \epsilon \quad \mbox{if $\epsilon<0$}.
	\end{align*}
	This is because when $\epsilon>0$, the $p_\epsilon(c)>p(c)$,
	so the merge height for $x_0,y_0$ using $p_\epsilon$ is still the same as $p(y_0)$,
	which implies $\xi_{p,\epsilon}(x_0,y_0) = d_{T_{p}}(x_0,y_0)$.
	On the other hand, when $\epsilon<0$, $p_\epsilon(c)<p(c)$,
	so the merge height is no longer $p(y_0)$ but $p_\epsilon(c)$.
	Then using the fact that $|\epsilon| = p(c) - p_\epsilon(c)$
	we obtain the result.
	
	Now we show that $ d_{T_{p}}(x_0,y_0)$ is not Hadamard differentiable.
	In this case, $\phi(p) = \xi_{p}(x_0,y_0)$.
	First, we pick a sequence of $\epsilon_n$ such that
	$\epsilon_n\rightarrow 0$ and $\epsilon_n>0$ if $n$ is even and $\epsilon_n<0$
	if $n$ is odd.
	Plugging $t\equiv \epsilon_n$ and $q_t = g$ into the definition of Hadamard differentiability,
	we have
	$$
	\phi'(p) \equiv \frac{\xi_{p,{\epsilon_n}}(x_0,y_0)-d_{T_{p}}(x_0,y_0)}{\epsilon_n}
	$$
	is alternating between $0$ and $2$,
	so it does not converge.
	This shows that the function $d_{T_{p}}(x,y)$ at such a pair of $(x_0,y_0)$ is non-Hadamard 
	differentiable.

	{\bf Case 2:}\\
	The proof of this case uses the similar idea as the proof of case 1.
	We pick the pair $(x_0,y_0)$ satisfying the desire conditions.
	We consider the same function $g $ but now we perturb $p$
	by
	$$
	p_\epsilon(x) = p(x) + \epsilon\cdot g(x-x_0),
	$$
	and as long as $\delta$ is small, we will have $p_\epsilon(y_0) = p(y_0)$.
	Since $B(x_0)=B(y_0)$ and $p(x_0)=p(y_0)$, $d_{T_{p}}(x_0,y_0)=0$.
	When $\epsilon>0$, $\xi_{p,\epsilon}(x_0,y_0) = \epsilon$,
	and on the other hand, when $\epsilon<0$, $\delta_\epsilon(x_0,y_0) = -\epsilon$.
	
	In this case, again, $\phi(p) = \xi_{p}(x_0,y_0)$.
	Now we use the similar trick as case 1: picking 
	a sequence of $\epsilon_n$ such that
	$\epsilon_n\rightarrow 0$ and $\epsilon_n>0$ if $n$ is even and $\epsilon_n<0$
	if $n$ is odd.
	Under this sequence of $\epsilon_n$, the `derivative' along $g$
	$$
	\phi'(p) \equiv\frac{\xi_{p,{\epsilon_n}}(x_0,y_0)-d_{T_{p}}(x_0,y_0)}{\epsilon_n}
	$$
	is alternating between $1$ and $-1$,
	so it does not converge.
	Thus, $d_{T_{p}}(x,y)$ at such a pair of $(x_0,y_0)$ is non-Hadamard 
	differentiable.
\end{proof}

\section{Proofs for Section \ref{sec:confidence_set} and Appendix \ref{app:confidence_set}}
\label{app:proof_confidence_set}

\subsection{Proof of Lemma \ref{lem:morse}}

\textbf{Lemma \ref{lem:morse}.} \textit{
	Let $p_h = \mathbb{E}[\hat p_h]$
	where $\hat p_h$ is the kernel estimator with bandwidth $h$.
	We assume that $p$ is a Morse function
	supported on a compact set with finitely many, distinct, critical values.
	There exists $h_0>0$ such that for all
	$0 < h < h_0$,
	$T_{p}$ and $T_{p_{h}}$ have the same topology in Appendix~\ref{app:topology}.
}

\begin{proof}
	Let $S$ be the compact support of $p$. By the classical
	stability properties of the Morse function, there exists a constant
	$C_{0}>0$ such that for any other smooth function $q:S\to\mathbb{R}$
	with $\|q-p\|_{\infty},\|\nabla q-\nabla p\|_{\infty},\|\nabla^{2}q-\nabla^{2}p\|_{\infty}<C_{0}$,
	$q$ is a Morse function. Moreover, there exist two diffeomorphisms
	$h:\mathbb{R}\to\mathbb{R}$ and $\phi:S\to S$ such that $q=h\circ p\circ\phi$
	See e.g., proof of \cite[Lemma 16]{chazal2014robust}.
	Further, $h$ should be nondecreasing if $C_{0}$ is small enough.
	Hence for any $C\in T_{p}(\lambda)$, since $q\circ\phi^{-1}(C)=h\circ p(C)$,
	so $\phi^{-1}(C)$ is a connected component of $T_{q}(h(\lambda))$.
	Now define $\Phi:\{T_{p}\}\to\{T_{q}\}$ as $\Phi(C)=\phi^{-1}(C)$.
	Then since $\phi$ is a diffeomorphism, $C_{1}\subset C_{2}$ if and
	only if $\Phi(C_{1})=\phi^{-1}(C_{1})\subset\phi^{-1}(C_{2})=\Phi(C_{2})$,
	hence $T_{p}\preceq T_{q}$ holds. And from $p\circ\phi=h^{-1}\circ q$,
	we can similarly show $T_{q}\preceq T_{p}$ as well. Hence from Lemma~\ref{lem:partial_equivalence}, two trees $T_{p}$ and $T_{q}$ are topologically equivalent according to the topology in Appendix~\ref{app:topology}.
	
	Now by the nonparametric theory (see e.g. page 144-145 of \cite{scott2015multivariate},
	and \cite{wasserman2006all}), there is a constant $C_{1}>0$ such
	that $\|p_{h}-p\|_{2,\max}\leq C_{1}h^{2}$ when $h<1$. Thus, when
	$0\leq h\leq\sqrt{\frac{C_{0}}{C_{1}}}$, $T_{h}=T_{p_{h}}$ and $T=T_{p}$
	have the same topology.
\end{proof}

\subsection{Proof of Lemma \ref{lem:pruning}}

\textbf{Lemma \ref{lem:pruning}.} \textit{
	Suppose that the $\life$ function satisfies: for all $[C]\in E(\hat{T}_{h})$,
	$\life^{top}([C])\leq\life([C])$. Then
	\begin{itemize}
		\item[(i)] $Pruned_{\life,\hat{t}_{\alpha}}(\hat{T}_{h}) \preceq \esttree$.
		\item[(ii)]	there exists a function $\tilde{p}$ such that $T_{\tilde{p}}=Pruned_{\life,\hat{t}_{\alpha}}(\hat{T}_{h})$.
		\item[(iii)] $\tilde{p}$ in (ii) satisfies $\tilde{p} \in \hat C_\alpha$.
	\end{itemize}
}

\begin{proof}
	
	(i)
	
	This is implied by Lemma \ref{lem:partial_insert}.
	
	(ii)
	
	Note that $Pruned_{\life,\hat{t}_{\alpha}}(\hat{T}_{h})$
	is generated by function $\tilde{p}$ defined as
	\[
	\tilde{p}(x)=\sup\left\{ \lambda:\,\text{there exists }C\in\hat{T}_{h}(\lambda)\text{ such that }x\in C\text{ and }\life([C])>2\hat{t}_{\alpha}\right\} +\hat{t}_{\alpha}.
	\]
	
	(iii)
	
	Let $C_{0}:=\bigcup\{C:\,\life([C])\leq2\hat{t}_{\alpha}\}$. Then
	note that 
	\[
	\hat{p}(x)=\sup\left\{ \lambda:\,\text{there exists }C\in\hat{T}_{h}(\lambda)\text{ such that }x\in C\right\} ,
	\]
	so for all $x$, $\tilde{p}(x)\leq\hat{p}(x)+\hat{t}_{\alpha}$, and
	if $x\notin C_{0}$, $\tilde{p}(x)=\hat{p}(x)+\hat{t}_{\alpha}$.
	Then note that 
	\begin{align*}
	& \left\{ \lambda:\,\text{there exists }C\in\hat{T}_{h}(\lambda)\text{ such that }x\in C\right\} \\
	& \backslash\left\{ \lambda:\,\text{there exists }C\in\hat{T}_{h}(\lambda)\text{ such that }x\in C\text{ and }\life([C])>2\hat{t}_{\alpha}\right\} \\
	& \subset\left\{ \lambda:\,\text{there exists }C\in\hat{T}_{h}(\lambda)\text{ such that }x\in C\text{ and }\life([C])\leq2\hat{t}_{\alpha}\right\} 
	\end{align*}
	Let $e_{x}:=\max\left\{ e:\,x\in\cup e,\,\life(e)\leq2\hat{t}_{\alpha}\right\} $.
	Then note that $x\in C$ and $\life([C])\leq2\hat{t}_{\alpha}$ implies
	that we can find some $B\in e_{x}$ such that $C\subset B$, so 
	\[
	\left\{ \lambda:\,\text{there exists }C\in\hat{T}_{h}(\lambda)\text{ such that }x\in C\text{ and }\life([C])\leq2\hat{t}_{\alpha}\right\} \subset cumlevel(e_{x}).
	\]
	Hence 
	\begin{align*}
	\hat{p}(x)+\hat{t}_{\alpha}-\tilde{p}(x) & \leq\sup\{cumlevel(e_{x})\}-\inf\left\{ cumlevel(e_{x})\right\} \\
	& =\life^{top}(e_{x})\\
	& \leq\life(e_{x})\leq2\hat{t}_{\alpha},
	\end{align*}
	and hence 
	\[
	\hat{p}(x)-\hat{t}_{\alpha}\leq\tilde{p}(x)\leq\hat{p}(x)+\hat{t}_{\alpha}.
	\]
\end{proof}

\end{document}